\documentclass[11pt]{article}
\usepackage{amsmath,amsthm,amssymb,a4wide} 
\usepackage[margin=2.5cm]{geometry}
\usepackage{graphicx}
\usepackage{epsfig} 
\usepackage{slashed}  
\usepackage{lmodern}
\usepackage[hidelinks]{hyperref}
\allowdisplaybreaks

\usepackage[T1]{fontenc}

\theoremstyle{plain}  
\newtheorem{theorem}{Theorem}[section]
\newtheorem{proposition}[theorem]{Proposition}
\newtheorem{lemma}[theorem]{Lemma}
\newtheorem{corollary}[theorem]{Corollary}
\newtheorem{definition}[theorem]{Definition}

\theoremstyle{remark}

\newtheoremstyle{remboldstyle}
  {}{}{}{}{\bfseries}{.}{.5em}{{\thmname{#1 }}{\thmnumber{#2}}{\thmnote{ (#3)}}}
\theoremstyle{remboldstyle}
\newtheorem{rembold}[theorem]{Remark}

\numberwithin{equation}{section} 
\numberwithin{figure}{section}  
\usepackage{amssymb, amscd, mathrsfs, wasysym} 

\let\OLDthebibliography\thebibliography
\renewcommand\thebibliography[1]{
  \OLDthebibliography{#1}
  \setlength{\parskip}{2pt}
  \setlength{\itemsep}{2pt plus 0.3ex}
}

\newcommand{\di}{\textrm{d}}

\newcommand \la \langle
\newcommand \ra \rangle

\newcommand \Ecal {\mathcal{E}} 
\newcommand \Kcal {\mathcal{K}}
\newcommand \Hcal {\mathcal{H}}

\newcommand \underdel {\underline \partial}

\newcommand \trianglerightNEW \triangleright 

\newcommand \hatL {\widehat{L}}

\newcommand \auth {\textsc} 

\newcommand \CC {\mathbb C}

\newcommand \bei {\begin{itemize}}
\newcommand \eei {\end{itemize}}
\newcommand \be {\begin{equation}}
\newcommand \bel {\be\label}
\newcommand \ee {\end{equation}}
\newcommand \bea {\be\aligned}
\newcommand \eea {\endaligned\ee}
\newcommand \del \partial
\newcommand \RR {\mathbb R}

\newcommand \eps \epsilon 

\newcommand \hL	{\widehat{L}}

\newcommand{\define}{:=}

\usepackage[most]{tcolorbox} 

\let\oldmarginpar\marginpar
\renewcommand\marginpar[1]{\-\oldmarginpar[\raggedleft\footnotesize #1]%
{\raggedright\footnotesize #1}}

\newcommand{\myfootnote}[1]{
    \renewcommand{\thefootnote}{}
    \footnotetext{\hspace{-16.5pt}\scriptsize#1}
    \renewcommand{\thefootnote}{\arabic{footnote}}
}

\begin{document}
\title{\bf \Large 
Asymptotic stability for the Dirac--Klein-Gordon system \\ in two space dimensions} 
\author{Shijie Dong${}^\ast$
and Zoe Wyatt${}^\dagger$
}
\date{}
\maketitle
 
\myfootnote{${}^\ast$Southern University of Science and Technology, SUSTech International Center for Mathematics, and Department of Mathematics, 518055 Shenzhen,  China. 
${}^\dagger$Department of Mathematics, King’s College London, Strand, London, WC2R 2LS. \\
{\sl Email addresses}: ${}^\ast$shijiedong1991@hotmail.com, dongsj@sustech.edu.cn ${}^\dagger$zoe.wyatt@kcl.ac.uk 
%{\sl MSC codes:} 35L05, 35L52, 35L70 \\
}

\begin{abstract} 
We study the Dirac--Klein-Gordon system in $1+2$ spacetime dimensions.  We show global existence of the solutions, as well as sharp time decay and linear scattering.  One key advance is that we provide the first asymptotic stability result for the Dirac--Klein-Gordon system in $1+2$ spacetime dimensions in the case of a massive Klein-Gordon field and a massless Dirac field. 
The nonlinearities are below-critical in two spatial dimensions, and so our method requires the identification of special structures within the system and novel weighted energy estimates. Another key advance, is that our proof allows us to weaken certain conditions on the nonlinear structures that have been assumed in the literature.
\end{abstract}

%==============================================================================================
\section{Introduction}
We consider the initial value problem for a coupled Dirac--Klein-Gordon system (DKG) in two spatial dimensions (2D). Following Bachelot \cite{Bachelot}, the general DKG system describes a spinor $\psi = \psi(t,x): \RR^{1+2} \to \CC^2$ of mass $M \in \RR$ and a scalar field $v = v(t,x):\RR^{1+2} \to \RR$ of mass $m \geq 0$, whose dynamics are governed by
\bel{eq:D-KG}
\aligned
-i \gamma^\mu \del_\mu \psi + M \psi
&= v F \psi,
\\
-\Box v + m^2 v 
&= \psi^* H \psi,
\endaligned
\ee
with prescribed initial data at $t=t_0=2$
\bel{eq:ID}
\big(\psi, v, \del_t v  \big)(t_0) = (\psi_0, v_0, v_1 ).
\ee
In the above DKG system, $H$ and $F$ are $2\times 2$ matrices with constant coefficients. 
The Dirac matrices $\{\gamma^0, \gamma^1, \gamma^2\}$ are a representation of the Clifford algebra, and are defined by the identities 
\be 
\aligned
\{ \gamma^\mu, \gamma^\nu \}
&\define \gamma^\mu \gamma^\nu + \gamma^\nu \gamma^\mu
= -2\eta^{\mu \nu} I_{2},
\quad
(\gamma^\mu)^* 
= -\eta_{\mu\nu} \gamma^\nu,
\endaligned
\ee
where $\mu, \nu \in \{ 0, 1, 2 \}$. 
Here, $I_2$ is the $2\times 2$ identity matrix, and $B^* = (\bar{B})^T$ denotes the Hermitian conjugate of a matrix $B$. 
%To be more concrete, the $\gamma$ matrices can be represented by 
%$$ 
%\gamma^0
%=
%\begin{pmatrix}
%1 & 0 \\
%0 & -1
%\end{pmatrix},
%\quad
%\gamma^1
%=
%\begin{pmatrix}
%0 & 1 \\
%-1 & 0
%\end{pmatrix}, 
%\quad
%\gamma^2
%=
%\begin{pmatrix}
%0 & -i \\
%-i & 0
%\end{pmatrix},
%$$
%however we do not use any explicit representation for the $\gamma$ matrices. The parameters $A, H$ are both $2\times 2$ constant matrices. 
We also define $\eta \define - \di t^2 + (\di x^1)^2 + (\di x^2)^2$ and use $\Box \define \eta^{\alpha\beta}\del_\alpha \del_\beta = -\del_t \del_t + \del_{x^1}\del_{x^1} + \del_{x^2}\del_{x^2}$ to denote the Minkowski wave operator. In the present paper, we study the case of a massless Dirac field and a massive scalar field, and without loss of generality we hereon set $\mathbf{M=0}$ and $\mathbf{m=1}$ unless otherwise specified. 

We define certain assumptions on the constant matrices $F, H$:
$$ \begin{array}{lllll}
\textbf{H1a} & F^* \gamma^0 = \gamma^0 F , & &  \textbf{H1b} & F = I_2, \\ \\ 
\textbf{H2a} &H^* = H,  & & \textbf{H2b} &H = \gamma^0.
\end{array}
$$
Conditions \textbf{H1a} and \textbf{H2a} are natural in the sense that \textbf{H1a} guarantees the conservation of charge
$$
\frac{d}{dt}\int_{\RR^2} \psi^* \psi = 0,
$$
while \textbf{H2a} ensures that the nonlinear term in the Klein-Gordon equation is real valued. Conditions \textbf{H1b} and \textbf{H2b} are, respectively, special cases of  \textbf{H1a} and \textbf{H2a}. We also note that there exist non-trivial examples of the matrix $F$ satisfying \textbf{H1a}, for instance $F=\gamma^\mu$ for $\mu = 0, 1, 2$.

\subsection{Main results}

We first state our main theorems and then discuss their relation to previous results in the literature, followed by an outline of the novel ideas used in our proof. 
\begin{theorem}\label{thm:main1}
Consider the initial value problem \eqref{eq:D-KG}--\eqref{eq:ID} under the assumptions $M=0, m=1$, \textbf{H1b} and \textbf{H2a},
and let $N \geq 7$ be an integer. There exists an $\eps_0 > 0$, such that for all $\eps \in (0, \eps_0)$, and all compactly supported initial data satisfying the smallness condition
\bel{thm:data-assumpt}
\|\psi_0\|_{H^{N}} + \| v_0 \|_{H^{N+1}} + \| v_1 \|_{H^{N}} \leq \eps,
\ee
the Cauchy problem \eqref{eq:D-KG}--\eqref{eq:ID} admits a global solution $(\psi, v)$. There exists a constant $C > 0$ such that the solution satisfies the following pointwise decay estimates 
\bel{eq:sharp-decay}
|\psi | \leq C \eps t^{-1/2} \big(1+|t-|x||\big)^{-1/2},
\qquad
|v|  \leq C \eps  t^{-1}.
\ee
Furthermore, the solution $(\psi, v)$ scatters linearly in the energy space.
\end{theorem}

\begin{theorem}\label{thm:main2}
Consider the initial value problem \eqref{eq:D-KG}--\eqref{eq:ID} under the assumptions $M=0, m=1$, \textbf{H1a} and \textbf{H2b}, and let $N \geq 4$ be an integer. There exists an $\eps_0 > 0$, such that for all $\eps \in (0, \eps_0)$, and all compactly supported initial data satisfying the smallness condition
\bel{thm:data-assumpt}
\|\psi_0\|_{H^{N}} + \| v_0 \|_{H^{N+1}} + \| v_1 \|_{H^{N}} \leq \eps,
\ee
the Cauchy problem \eqref{eq:D-KG}--\eqref{eq:ID} admits a global solution $(\psi, v)$. There exists a constant $C > 0$ such that the solution satisfies the following pointwise decay estimates 
\bel{eq:sharp-decay}
|\psi | \leq C \eps t^{-1/2} \big(1+|t-|x||\big)^{-1/2},
\qquad
|v|  \leq C \eps  t^{-1}.
\ee
Furthermore, the solution $(\psi, v)$ scatters linearly in the energy space.
\end{theorem}

\begin{rembold}
In both Theorem \ref{thm:main1} and Theorem \ref{thm:main2}, the pointwise decay of the solutions is sharp in time in the sense that the solutions enjoy the same decay rates in time as the linear equations. Thus, we prove asymptotic stability for the 2D DKG system \eqref{eq:D-KG} under the relevant assumptions stated in the theorems. Indeed, our result provides the \textbf{first asymptotic stability result} for the DKG system for the case $M=0, m=1$ for smooth, small, and compactly supported initial data. 
\end{rembold}

\begin{rembold}
 Gr\"{u}nrock and Pecher \cite{GH10} have shown global existence for the 2D DKG system \eqref{eq:D-KG} under the assumptions $M, m \in \mathbb{R}$, \textbf{H1b} and \textbf{H2b} and with (large) low-regularity data
$$ \psi_0 \in L^2(\RR^2), \quad v_0 \in H^{1/2}(\RR^2), \quad v_1 \in H^{-1/2}(\RR^2). $$
Thus, our main contribution for the case $M=0, m = 1$ is to show asymptotic stability and to \textbf{weaken the structural assumptions}  on the nonlinearities considered in \cite{GH10}. In particular in Theorem \ref{thm:main1} we can allow for $H \neq \gamma^0$. It is not yet clear whether the most general case of \textbf{H1a} and \textbf{H2a} can be shown to admit small global solutions.
\end{rembold}

\begin{rembold}
At present, most global existence and decay results for 2D coupled wave and Klein-Gordon equations restrict their analysis to the interior of a lightcone (i.e. the data is assumed to be compact). There is some work that does not require this, see for instance \cite{Dong2005, Stingo}. In 3D, there are methods which treat  both the interior and exterior regions of a lightcone for the Maxwell--Klein-Gordon equations \cite{FangWangYang} (see also the work \cite{Klainerman-QW-SY} concerning the exterior region). 
It is, however, not yet clear to us whether these methods can be used to remove our compactness assumptions.
\end{rembold}

\subsection{Previous work on the DKG system}\label{sec:prevDKGwork}
The system \eqref{eq:D-KG}  
arises in particle physics as a model for Yukawa interactions between a scalar field and a Dirac spinor. It appears in the theory of pions and in the Higgs mechanism \cite{physics}. We note that the nonlinearity $\psi^*\gamma^0\psi$ is often writen as $\bar{\psi}\psi$ where $\bar{\psi}\define \psi^*\gamma^0$ is the Dirac adjoint, and thus transforms as a scalar under Lorentz transformations. The Cauchy problem for the DKG system has been actively studied in various spacetime dimensions and for different cases of the Klein-Gordon and Dirac masses (i.e. $m \geq 0$ and $M \geq 0$). \vspace{5pt}

\noindent \textbf{Three spatial dimensions.} For high-regularity initial data, there are small-data  results that show global existence for certain subcases of \eqref{eq:D-KG} with asymptotic decay rates \cite{Bachelot, Katayama12a}. Similar results are also known for the closely related Dirac--Proca system \cite{Tsutsumi, Katayama12a}. 
For low-regularity initial data, the problem is more difficult as the  natural energy density associated to these DKG systems  does not have a definite sign. The lack of positive definite conserved quantities makes it particularly difficult to prove global existence and scattering for low-regularity data. For results (note under conditions \textbf{H1b} and \textbf{H2b}), see for example  \cite{Bejenaru-Herr, WangX} and references within, as well as for large-data results, see for example \cite{Dias-Fig91, Candy-Herr2, DFS-07} and references cited within. 
\vspace{5pt}

\noindent \textbf{Two spatial dimensions.}
For high-regularity initial data, global existence and asymptotic stability to the DKG system \eqref{eq:D-KG} for the case $M>0, m >0$ was shown in the works \cite{ST93, OTT2} for smooth, small initial data. These results rely on transforming the DKG system \eqref{eq:D-KG} into two coupled Klein-Gordon equations. The asymptotic stability (even the stability) result for the other cases of $M=1, m=0$ or $M=m=0$ remains open however. For low-regularity initial data, there are local  existence results under the assumptions \textbf{H1b} and \textbf{H2b} \cite{Bournaveas-2D, DFS-07-2D}. Global  existence for low-regularity and possibly large data, again under the conditions \textbf{H1b} and \textbf{H2b}, is known by  \cite{GH10}. 
The study of two dimensional Dirac equations \cite{KatayamaKubo, Bejenaru-Herr2D, Li-Zang} is also relevant to our study. %Finally, we %we note that chiral decompositions for $\psi$, used for example in \cite{DLW, Li-Zang} in 3D, are not possible in 2D. 

\subsection{Major difficulties and challenges}
We first remind the reader of the important identity
\bel{eq:dirac-to-wave}
\Box \psi = \big(i \gamma^\mu \del_\mu \big) \big( i \gamma^\nu \del_\nu \psi \big).
\ee
Thus we can think of \eqref{eq:D-KG} as encoding a coupled wave-like and Klein-Gordon system. 
Proving global existence and asymptotic decay results for coupled nonlinear wave and Klein-Gordon equations, such as in Theorems \ref{thm:main1} and \ref{thm:main2}, is typically a challenging question in two spatial dimensions. This is because linear wave $w$ and linear Klein-Gordon $v$ equations have very slow pointwise decay rates in $\RR^{1+2}$, namely 
\bel{intro-linear-decay}
|w| \lesssim \big(1+t + |x|\big)^{-1/2} \big(1 + |t-|x||\big)^{-1/2}, 
\qquad
|v| \lesssim \big(1+t+|x|\big)^{-1}.
\ee
The identity \eqref{eq:dirac-to-wave} also indicates that a linear massless Dirac field should obey the same slow pointwise decay rates as $|w|$ above. 
As a consequence of \eqref{intro-linear-decay}, when using Klainerman's vector field method \cite{Klainerman86} on quadratic nonlinearities, we might \emph{at best} get an integral of $t^{-1}$. This leads to problems when closing the bootstrap argument and can possibly indicate finite-time blow-up.

Another obstacle when studying Klein-Gordon equations, in the framework of the vector field method, is that the scaling vector field $L_0 = t\del_t + x^1 \del_{x^1} + x^2 \del_{x^2}$ does not commute with the Klein-Gordon operator $-\Box + 1$. The scaling vector field can be avoided by using a spacetime foliation of surfaces $\Hcal_s$ of constant hyperboloidal time $s = \sqrt{t ^2 - |x|^2}$. This idea originates in work of Klainerman \cite{Klainerman85, Klainerman93} (see also  H\"ormander \cite{Hormander}) on  Klein-Gordon equations, and was later reintroduced to
%establish global-in-time existence results for nonlinear systems of coupled wave and Klein-Gordon equations 
treat coupled wave and Klein-Gordon equations by LeFloch and Ma in \cite{PLF-YM-book}  under the name of the ``hyperboloidal foliation method''. This method can be regarded as Klainerman's vector field method on hyperboloids. 
We also remind the reader of the pioneering work by Tataru showing Strichartz estimates for wave equations in the hyperbolic space \cite{Tataru}, and the work by Psarelli \cite{Psarelli} on the Maxwell--Klein-Gordon equations.  

%\marginpar{Possibly remove this paragraph?}
%This method, however, prevents one from using the classical Klainerman--Sobolev inequality to gain additional $(t-|x|)$-decay. Furthermore the inability to control the $L_0$ vector field on a Klein-Gordon field leads to issues, in $\RR^{1+2}$, when controlling the symmetric quadratic null form between two Klein-Gordon components (see the discussion in \cite[\textsection 3.3]{DW-2020}).

Returning now to the DKG problem \eqref{eq:D-KG}, we use the identity \eqref{eq:dirac-to-wave} to derive the following
\bel{eq:W-KG}
\aligned
-\Box \psi &= i \gamma^\nu \del_\nu \big( v F \psi  \big),
\quad
-\Box v + v =\psi^* H \psi.
\endaligned
\ee
If we ignore the structure here (indeed, under \textbf{H2a} the term $\psi^* H \psi$ does not have any special structure), we roughly speaking have obtained a wave--Klein-Gordon system  of the form
\bel{eq:simplified-wKG} 
\aligned
-\Box w &= \del (v w) =  w \del v+ v \del w,
\quad
-\Box v + v = w^2.
\endaligned
\ee
The global existence of general small-data solutions to \eqref{eq:simplified-wKG} is presently unknown in $\RR^{1+2}$. Furthermore, if we assume that $w$ and  $v$ obey the linear estimates \eqref{intro-linear-decay}, then the best we can expect from the nonlinearities (in the flat $t=\text{cst.}$ slices) is 
$$
\|\del (v w) \|_{L^2(\RR^2)} \lesssim t^{-1}, \qquad \| w^2 \|_{L^2(\RR^2)} \lesssim t^{-1/2}.
$$ 
Returning to the original PDE \eqref{eq:D-KG}, for example under the assumptions \textbf{H1b} and \textbf{H2a} of Theorem \ref{thm:main1}, the best we can expect appears to be
$$
\|v \psi\|_{L^2(\RR^2)} \lesssim t^{-1}, \qquad \|\psi^* H\psi \|_{L^2(\RR^2)} \lesssim t^{-1/2}.
$$ 
Thus one quantity is at the borderline of integrability and the other is strictly \emph{below} the borderline of integrability. In previous work of the authors \cite{DW-2020}, such a situation was termed `below-critical' in time decay, and indicates that if the classical vector field method is to be successful, then  new ideas are required to close both the lower and higher order bootstraps.

\subsection{Key ingredients and new ideas} 
To conquer the aforementioned difficulties in studying the DKG equations \eqref{eq:D-KG}, we need several ingredients and novel observations that go beyond classical methods for Klein-Gordon equations such as in \cite{Klainerman85, Klainerman93}. 
%If we ignore the structure of the DKG nonlinearities and study the more general wave-Klein-Gordon system \eqref{eq:simplified-wKG}, then, as explained before, we immediately face issues coming from the below-critical nonlinearities which could even lead to finite-time blow-up. Given this, we instead directly study the system \eqref{eq:D-KG} and uncover hidden structure within the nonlinearities. 
The first ingredient is an energy functional, defined on hyperboloids, for solutions to the Dirac equation. This was first derived by the authors and LeFloch in \cite{DLW}. Using this Dirac-energy functional, we find that the best behaviour we can hope for is
$$
\aligned
\| (s/t) \psi \|_{L^2_f(\Hcal_s)} &\lesssim 1,
\qquad
&|\psi| &\lesssim t^{-1/2} (t-|x|)^{1/2} \lesssim s^{-1},
\\
\| v \|_{L^2_f(\Hcal_s)} &\lesssim 1,
\qquad
&|v| &\lesssim t^{-1}.
\endaligned
$$
Here $\Hcal_s$ are constant $s$-surfaces defined in Section \ref{sec:pre-hyp} and $L^2_f(\Hcal_s)$ is defined in \eqref{eq:L-1-f}.
Rough calculations, for instance under the assumptions \textbf{H1b} and \textbf{H2a},  lead us to the estimates  
\begin{align}
\|v \psi\|_{L^2_f(\Hcal_s)} 
&\lesssim \| (s/t) \psi \|_{L^2_f(\Hcal_s)} \| (t/s) v\|_{L^\infty(\Hcal_s)}
\lesssim s^{-1},\notag
\\
\| \psi^* H \psi \|_{L^2_f(\Hcal_s)} 
&\lesssim \|(s/t)\psi\|_{L^2_f(\Hcal_s)} \|(t/s)\psi\|_{L^\infty(\Hcal_s)} 
\lesssim 1.\label{intro_psiHpsi}
\end{align}
We see that one term is at, and the other is below, the borderline of integrability. We remark that the only other known work in the literature of coupled wave and Klein-Gordon equations studying such a situation is in our \cite{DW-2020}.

Our first new insight is to notice that a field can be thought of as `Klein-Gordon type' if its $L^2_f(\Hcal_s)$-norm is well-controlled by the natural energy functionals. We know that examples of `Klein-Gordon type' fields include $v, (s/t) \del_\alpha v$ and we discover the further examples
$$
(s/t)\psi,
\quad
\psi - (x^a/t) \gamma^0 \gamma^a \psi.
$$

We then uncover a decomposition (see Lemma \ref{lem:hidden-KG}) in the following Dirac--Dirac interaction term
\bel{eq:intro-KG-type-factor}
\aligned
\psi^* \gamma^0 \psi 
&\sim  \big( \psi - \tfrac{x^a}{t} \cdot \gamma^a \psi\big) ^*\big( \psi - \tfrac{x^a}{t} \gamma^0 \gamma^a \psi\big) 
+ \big(\tfrac{s}{t}\psi \big)^*\big(\tfrac{s}{t}\psi \big)
 + \big(\psi - \tfrac{x^a}{t} \gamma^0 \gamma^a \psi\big)^* \psi.
\endaligned
\ee
The key observation is that terms on the RHS above always involve at least one `Klein-Gordon type' factor. This observation is of vital importance in the proof of both Theorems \ref{thm:main1} and \ref{thm:main2}. For example when $H = \gamma^0$, \eqref{eq:intro-KG-type-factor} allows us to improve the initial estimate given in \eqref{intro_psiHpsi} to 
\[
\| \psi^* \gamma^0 \psi \|_{L^2_f(\Hcal_s)} 
\lesssim \| \psi - \tfrac{x^a}{t} \cdot \gamma^a \psi\|_{L^2_f(\Hcal_s)} \|\psi\|_{L^\infty(\Hcal_s)} + \ldots 
\lesssim s^{-1} .
\]
%When $H$ is not necessarily equal to $\gamma^0$, we find that by the multiplier method for the Dirac equation  we nevertheless still need to estimate the nonlinearity $\psi^* \gamma^0(\gamma^\mu \partial_\mu \psi)$ (see Proposition \ref{prop:energy-ineq-Dirac}).
Interestingly, we find that several other Dirac-Dirac interactions, such as $\psi^*\psi$, do not possess the same useful decomposition (see Remark \ref{rem:hidden-KG}). In addition, we  find that the structure of the nonlinearity $\psi^*\gamma^0\psi$
is preserved under commutation with the Lorentz boosts (see Lemma \ref{lem:L-Dirac}) and thus the decomposition \eqref{eq:intro-KG-type-factor} can be applied at higher orders.

The next ingredient comes from using nonlinear transformations to remove slowly-decaying nonlinearities (see Lemma \ref{lem:transf-v} and a new transformation for the Dirac field given in Lemma \ref{lemma:EoM-tilde-psi}) when estimating the low order energy. This comes at the expense of introducing cubic nonlinearities and quadratic null forms and we are able to close the bootstrap at lower-orders, provided we can control these null forms. 

One more ingredient, needed to control the null forms introduced in the previous paragraph, is to obtain additional $(t-r)$-decay for the Dirac spinor. 
In the case of pure wave equations it is well-known that one can obtain extra $(t-r)$-decay with the aid of the full range of vector fields $\{\del_\alpha, \Omega_{ab}, L_a, L_0  \}$ (defined in Section \ref{sec:pre-hyp}). For instance, for sufficiently regular functions $\phi$ we have the estimate \cite{Sogge}
\bel{eq:extra-001}
\big| \del \del \phi \big|
\lesssim \big( 1+ |t-r| \big)^{-1} \big( \big|L_0 \del \phi \big| + \sum_a \big| L_a \del \phi \big|  \big).
\ee
If we cannot control certain vector fields acting on our solution, then it is usually more difficult to obtain extra $(t-r)$-control as in \eqref{eq:extra-001}. We recall two examples of similar situations: 1) obtaining extra $(t-r)$-decay in the case of nonlinear elastic waves by Sideris \cite{Sideris} where Lorentz boosts $L_a = t\del_a + x_a \del_t$ are unavailable; 2) obtaining extra $(t-r)$-decay in the case of coupled wave--Klein-Gordon equations by LeFloch-Ma \cite[\textsection8.1, \textsection8.2]{PLF-YM-book} where the scaling vector field $L_0 $ is absent. 

In the DKG model \eqref{eq:D-KG} we also cannot use $L_0$, nor can we directly gain $(t-r)$-decay by studying the wave equation in \eqref{eq:W-KG}. Our insight, inspired by the work \cite{PLF-YM-book}, is to rewrite the Dirac operator in a frame adapted to the hyperboloidal foliation. The latter idea yields the following estimate
$$
\big|\del_t \psi \big|
\lesssim
{1\over t-r}\sum_a |L_a \psi | + {t\over t-r}|i\gamma^\mu \del_\mu \psi |.
$$
This argument gives us the extra $(t-r)$-decay for $\del \psi$ (see Lemma \ref{lem:extra-t-r2} and Proposition \ref{prop:psi-t}) required to close the null form estimates. 

The final ingredient, key to closing the highest order bootstrap for Theorem \ref{thm:main1}, is to derive weighted energy inequalities. We recall that we cannot rely on nonlinear transformations when estimating the highest order energy, and the nonlinearities are below-critical. Our idea is to derive and rely on a $(t-r)$-weighted Dirac energy functional (see Proposition \ref{prop:GhostWeight}). Such weighted estimates were introduced in \cite{Alinhac01b}, and have recently been adapted to the hyperboloidal setting, with applications to the Klein-Gordon--Zakharov system in \cite{Dong2005}. We utilise such weighted estimates here for the first time for Dirac equations. 

\begin{rembold}
We expect the ideas in the proof of Theorems \ref{thm:main1} and \ref{thm:main2} to have other applications. For instance, it can be used to show uniform energy bounds for the solution to the 3D Dirac-Klein-Gordon equations studied by Bachelot in \cite{Bachelot} as well as the $U(1)$-Higgs model studied in \cite{DLW}.
\end{rembold}
\begin{rembold}
Our decomposition approach in \eqref{eq:intro-KG-type-factor} in fact gives a reinterpretation of structure identified by Bournaveas \cite{Bournaveas-JFA, Bournaveas-2D}. Suppose that there exists $\phi$ such that $\psi = i \gamma^\nu \del_\nu \phi$. Using \eqref{eq:dirac-to-wave} one can show that  $- \Box \phi = i \gamma^\nu v F  \del_\nu \phi$
and also 
\bel{eq:null-gamma0-structure}
\psi^*\gamma^0\psi = \left(  (\del_t \phi)^* \del_t (\gamma^0 \phi) - \delta^{ij} (\del_i \phi)^*  \del_j(\gamma^0\phi) \right) + \left( -(\del_i \phi)^*  \del_t(\gamma^i\phi) + (\del_t \phi)^*  \del_i(\gamma^i\phi)\right)
\ee
The two bracketed terms in \eqref{eq:null-gamma0-structure} are semilinear null terms, which are known to obey better estimates (see for example Lemma \ref{lem:null}). Such null structure played an essential role in the previous works, for example  \cite{GH10}, mentioned in section \ref{sec:prevDKGwork} that rely on \textbf{H2b}. In the case of Theorem \ref{thm:main1}, however, our approach allows us to weaken the assumption of \textbf{H2b} to \textbf{H2a}. 
\end{rembold}

\subsection{Wave--Klein-Gordon Literature}
To conclude the introduction, we remind the reader of some of the literature concerning global existence and decay for coupled wave--Klein-Gordon equations. In 3D these include wave--Klein-Gordon equations  derived from mathematical physics, such as the Dirac-Klein-Gordon model, the Dirac-Proca and $U(1)$-electroweak model \cite{DLW, Katayama12a, Tsutsumi}, the Einstein-Klein-Gordon equations \cite{Ionescu-P-EKG, PLF-YM-cmp, PLF-YM-arXiv1, Wang}, the Klein-Gordon-Zakharov equations \cite{OTT}, the Maxwell-Klein-Gordon  equations \cite{Klainerman-QW-SY, FangWangYang}  and certain geometric problems derived from wave maps \cite{Abbrescia-Chen}.  

Very recently, there has been much research concerning global existence and decay for wave--Klein-Gordon equations in 2D. 
%This was initiated by Ma for quasilinear wave-Klein-Gordon systems \cite{Ma2017a, Ma2017b} and has been succeeded by Ma \cite{Ma2019, Ma2020} and the present authors \cite{DW-2020}. 
We mention for instance the works by Ma \cite{Ma2020} and the present authors \cite{DW-2020} for compactly supported initial data; see also the references therein.
There have also been works \cite{Stingo, IStingo, Dong2005} that have investigated wave and Klein-Gordon systems under certain null conditions without the restriction to compactly supported data. 
Other work has looked at the Klein-Gordon-Zakharov model in $1+2$ dimensions \cite{Dong2006, DM20, Ma2020, DongMa} and the wave map model derived in \cite{Abbrescia-Chen} has been studied in the critical case of $1+2$ dimensions in the recent works \cite{DW-2021, DM20} (see also \cite{Wong}). An analysis of general classes of cubic nonlinearities has also been given in \cite{C21}. \newline

\noindent\textbf{Outline.}
We organise the rest of the paper as follows. In Section \ref{sec:pre} we introduce some essential notation and the preliminaries of the hyperboloidal method. In Section \ref{sec:hidden} we present the essential hidden structure within the nonlinearities. Finally, Theorems \ref{thm:main1} and \ref{thm:main2} are proved in Section \ref{sec:proof} and Appendix \ref{sec:proof2}, respectively, by using a classical bootstrap argument. 

%==============================================================================================
\section{Preliminaries}\label{sec:pre}

\subsection{Basic notations}\label{sec:pre-hyp}

We denote a spacetime point in $\RR^{1+2}$ by $(t, x) = (x^0, x)$, and its spatial radius by $r\define \sqrt{(x^1)^2 + (x^2)^2}$. Following Klainerman's vector field method \cite{Klainerman86}, we introduce the following vector fields
\begin{align*}
\del_\alpha &\define\del_{x^\alpha}, \quad L_a \define t\del_a + x_a \del_t , \quad 
\Omega_{ab} \define x_a \del_b - x_b \del_a , \quad
L_0 \define t\del_t + x^a \del_a.
\end{align*}
Such vector fields are referred to as translations, Lorentz boosts, rotations and scaling respectively. 
We also use the modified Lorentz boosts, first introduced by Bachelot \cite{Bachelot},
$$ 
\widehat{L}_a \define L_a - \tfrac12 \gamma^0 \gamma^a.
$$
These are chosen to be compatible with the Dirac operator, in the sense that  $[\widehat{L}_a , i\gamma^\mu \del_\mu] = 0$
where we have used the standard notation for commutators $[A, B] \define AB - BA$. 

We restrict out study to functions supported within the spacetime region $\Kcal \define \{ (t, x) : t \geq 2, \, t \geq |x| + 1 \}$ which we  foliate using hyperboloids.
A hyperboloid $\Hcal_s$ with hyperboloidal time $s \geq s_0 =2$ is defined by $\Hcal_s \define \{ (t, x) : t^2 = |x|^2 + s^2 \}$. We find that any point $(t, x) \in \Kcal \cap \Hcal_s$ with $s\geq 2$ obeys the following relations
$$
|x| \leq t,
\qquad
s\leq t \leq s^2.
$$
Without loss of generality we take $s_0 = 2$, and we use $\Kcal_{[s_0, s_1]} \define \bigcup_{s_0 \leq s \leq s_1} \Hcal_s \bigcap \Kcal$ to denote the spacetime region between two hyperboloids $\Hcal_{s_0}, \Hcal_{s_1}$.
We follow LeFloch and Ma \cite{PLF-YM-book} and introduce the semi-hyperboloidal frame
$$ 
\underdel_0 \define \del_t,
\qquad
\underdel_a \define t^{-1} L_a = {x_a \over t} \del_t + \del_a.
$$
The semi-hyperboloidal frame is adapted to the hyperboloidal foliation setting since the set $\underdel_a$ generate the tangent space to the hyperboloids. 
The usual partial derivatives, i.e. those in a Cartesian frame, can  be expressed in terms of the semi-hyperboloidal frame as
$$
\del_t = \underdel_0,
\qquad
\del_a = -{x_a \over t} \del_t + \underdel_a.
$$

\noindent\textbf{Standard notation.} We use $C$ to denote a universal constant, and $A\lesssim B$ to indicate the existence of a constant $C>0$ such that $A\leq BC$.  For the ordered sets $\{ Z_i\}_{i=1}^5\define\{ \del_{0},\del_1, \del_2, L_1,L_2 \}$, $\{ \widehat{Z}_i\}_{i=1}^5\define\{ \del_{0},\del_1, \del_2, \hL_1, \hL_2 \}$, and for any multi-index $I=(\alpha_1, \ldots, \alpha_5)$ of length $|I|\define \alpha_1+ \cdots + \alpha_5$ we denote by $Z^I = Z_1^{\alpha_1} \cdot \cdots\cdot Z_5^{\alpha_5}$ and $\widehat{Z}^I = \widehat{Z}_1^{\alpha_1} \cdot \cdots\cdot \widehat{Z}_5^{\alpha_5}$. Spacetime indices are represented by Greek letters while spatial indices are denoted by Roman letters. We adopt Einstein summation convention unless otherwise specified. We will often write $|\del \phi|$, respectively $|\underdel \phi|$, to denote an estimate on $|\del_\mu \phi|$ for arbitrary $\mu$, respectively $|\underdel_a \phi|$ for arbitrary $a$.

\subsection{Energy estimates for wave and Klein-Gordon fields on hyperboloids}

Given a function $\phi = \phi(t, x)$ defined on a hyperboloid $\Hcal_s$, we define its $\| \cdot \|_{ L^1_f(\Hcal_s)}$ norm as
\bel{eq:L-1-f}
\| \phi \|_{ L^1_f(\Hcal_s)}
= \int_{\Hcal_s} |\phi(t, x)| \, dx
\define \int_{\RR^2} |\phi(\sqrt{s^2 + |x|^2}, x)| \, dx.
\ee
With this, the norm $\| \cdot \|_{ L^p_f(\Hcal_s)}$ for $1\leq p < +\infty$ can be defined. The subscript $f$ comes from that the fact that the volume form in \eqref{eq:L-1-f} comes from the standard flat metric in $\RR^2$. 

Following \cite{Hormander, PLF-YM-book}, we define the following $L^2$-based energy of a function $\phi = \phi(t, x)$, scalar-valued or vector-valued, on a hyperboloid $\Hcal_s$
$$
\aligned
\Ecal_m (s, \phi)
&\define
\int_{\Hcal_s} \Big( \sum_\alpha |\del_\alpha \phi|^2 + {x^a \over t} \big( \del_t \phi^* \del_a \phi + \del_a \phi^* \del_t \phi \big) + m^2 |\phi|^2 \Big) \, dx
\\
&=
\int_{\Hcal_s} \Big(|(s/t) \del_t \phi|^2 + \sum_a |\underdel_a \phi|^2 + m^2 |\phi|^2 \Big) \, dx
\\
&=
\int_{\Hcal_s} \Big( \sum_a |(s/t) \del_a \phi|^2 + t^{-2}|\Omega_{12} \phi|^2 + t^{-2} |L_0 \phi|^2  + m^2 |\phi|^2 \Big) \, dx.
\endaligned
$$
Note in the above  $m \geq 0$ is a constant. 
From the last two equivalent expressions of the energy functional $\Ecal_m$, we easily obtain
$$
\sum_\alpha \|(s/t) \del_\alpha \phi\|_{L^2_f(\Hcal_s)}
+ \sum_a \| \underdel_a \phi\|_{L^2_f(\Hcal_s)}
\leq C \Ecal_m (s, \phi)^{1/2}.
$$
We also adopt the abbreviation $\Ecal (s, \phi) = \Ecal_0 (s, \phi)$.
We have the following classical energy estimates for wave and Klein-Gordon equations.

\begin{proposition}\label{prop:energy-ineq-KG}
Let $\phi$ be a sufficiently regular function with defined in the region $\Kcal_{[s_0, s_1]}$ and vanishing near $\del\Kcal_{[s_0, s_1]}$. Then, for all $s \in [s_0, s_1]$ we have
$$
\Ecal_m (s, \phi)^{1/2}
\leq \Ecal_m (s_0, \phi)^{1/2} + \int_{s_0}^s \| -\Box \phi + m^2 \phi \|_{L^2_f(\Hcal_\tau)} \, \di \tau.
$$
\end{proposition}

\subsection{Energy estimates for Dirac fields on hyperboloids}

Let $\Psi(t,x): \RR^{1+2} \to \CC^2$ be a complex-valued function defined in the region $\Kcal_{[s_0, \infty)}$. We introduce the energy functionals
\bel{eq:D-fctnal-1} 
\aligned
\Ecal^+ (s, \Psi)
&\define\int_{\Hcal_s} \Big( \Psi - {x^a \over t} \gamma^0 \gamma^a \psi \Big)^* \Big( \Psi - {x^a \over t} \gamma^0 \gamma^a \Psi  \Big) \, \di x.
\\
\Ecal^D (s, \Psi)
&\define
\int_{\Hcal_s} \Big( \Psi^* \Psi - {x^a \over t} \Psi^* \gamma^0 \gamma^a \Psi \Big) \, \di x
\endaligned
\ee
These were first introduced in \cite{DLW}, and the following useful identity was also derived
\bel{eq:D-fctnal-2} 
\Ecal^D (s, \Psi)
=
{1\over 2} \int_{\Hcal_s} {s^2 \over t^2} \Psi^* \Psi \, \di x + {1\over 2} \Ecal^+ (s, \Psi). 
\ee
From this identity we obtain the non-negativity of the functional $\Ecal^D (s, \Psi)$ and the inequality
$$
\Big\|\frac{s}{t} \Psi\Big\|_{L^2_f(\Hcal_s)}
+ \Big\|\Big (I_2 - \frac{x^a}{t} \gamma^0 \gamma^a\Big)\Psi \Big\|_{L^2_f(\Hcal_s)}
\leq C \Ecal^D (s, \Psi)^{1/2}.
$$

We have the following energy estimates (see \cite[Prop. 2.3]{DLW} for \eqref{eq:D-E1} and see \cite{DLMY} for an application of \eqref{eq:D-E2}).

\begin{proposition}\label{prop:energy-ineq-Dirac}
Let $\Psi(t,x): \RR^{1+2} \to \CC^2$ be a sufficiently regular  function with support in the region $\Kcal_{[s_0, s_1]}$. Then for all $s \in [s_0, s_1]$ we have
\begin{align}
\label{eq:D-E1}
\Ecal^D (s, \Psi)^{1/2}
&\leq \Ecal^D (s_0, \Psi)^{1/2} + \int_{s_0}^s \| i \gamma^\mu \del_\mu \Psi \|_{L^2_f(\Hcal_\tau)} \, \di \tau,
\\
\label{eq:D-E2}
\Ecal^D (s, \Psi)
&\leq \Ecal^D (s_0, \Psi) + 2\int_{s_0}^s \|i \Psi^* \gamma^0 \gamma^\mu \del_\mu \Psi  \|_{L^1_f(\Hcal_\tau)} \, \di \tau.
\end{align}
\end{proposition}

\subsection{Weighted energy estimates}
Following ideas of Alinhac \cite{ Alinhac01b}, we next derive weighted energy estimates. These have been applied to coupled wave--Klein-Gordon systems in \cite[Prop. 3.2]{Dong2006}, and here we pursue similar estimates for Dirac equations.

We first define the $(t-r)$--weighted energy for a Dirac field
\bel{eq:D-fctnal-1a} 
\aligned
\Ecal^+ (s, \Psi, \gamma)
&\define\int_{\Hcal_s} (t-r)^{-2\gamma} \Big( \Psi - {x^a \over t} \gamma^0 \gamma^a \psi \Big)^* \Big( \Psi - {x^a \over t} \gamma^0 \gamma^a \Psi  \Big) \, \di x.
\\
\Ecal^D (s, \Psi, \gamma)
&\define
\int_{\Hcal_s} (t-r)^{-2\gamma} \Big( \Psi^* \Psi - {x^a \over t} \Psi^* \gamma^0 \gamma^a \Psi \Big) \, \di x
\endaligned
\ee
The following useful identity holds:
\bel{eq:D-fctnal-2a} 
\Ecal^D (s, \Psi, \gamma)
=
{1\over 2} \int_{\Hcal_s} {s^2 \over t^2} (t-r)^{-2\gamma} \Psi^* \Psi \, \di x + {1\over 2} \Ecal^+ (s, \Psi, \gamma). 
\ee

\begin{proposition}\label{prop:GhostWeight}
For $M\geq 0$ consider a sufficiently regular function $\Psi$ defined in the region $\Kcal_{[s_0, s]}$, vanishing near $\del\Kcal_{[s_0, s]}$ and satisfying
\be\notag
-i\gamma^\mu \del_\mu \Psi + M \Psi = f.
\ee
Then for $\gamma>0$ we have
\bel{eq:GhostDirac} 
\Ecal^D (s, \Psi, \gamma)
\leq C \Ecal^D (s_0, \Psi, \gamma)
+ C \int_{s_0}^s \big\| (t-r)^{-2\gamma} \Psi^* \gamma^0 f  \big\|_{L^1_f(\Hcal_\tau)} \, \di \tau.
\ee

%\bea\notag
%\int_{\Hcal_s} &(t-r)^{-\gamma}\Big[ (\del_s \phi)^2 + \sum_{a=1}^2(\underdel_a \phi)^2 + m^2 \phi^2 \Big] \di x 
%\\
%&
%\leq \int_{\Hcal_{s_0}} (t-r)^{-\gamma}\Big[ (\del_s \phi)^2 + \sum_{a=1}^2(\underdel_a \phi)^2 + m^2 \phi^2 \Big] \di x 
%+ 2 \int_{s_0}^s \int_{\Hcal_\tau} (\tau/t) (t-r)^{-\gamma} f \del_t \phi \di x \di\tau.
%\eea
%\bea\notag
%&\int_{\Hcal_s} \big(t-r\big)^{-\gamma}\Big[ \big((s/t)\del_t \phi\big)^2 + \sum_a (\underdel_a \phi)^2 + m^2 \phi^2 \Big] \di x 
%\\
%&
%\leq C \int_{\Hcal_{s_0}}\big(t-r\big)^{-\gamma} \Big[ ((s/t)\del_t \phi)^2+ \sum_a (\underdel_a \phi)^2 + m^2 \phi^2 \Big] \di x 
%+ C \int_{s_0}^s \big\|(\tau/t) (t-r)^{-\gamma} f \del_t \phi \di x \big\|_{L^1_f(\Hcal_\tau)} \di\tau .
%\eea
\end{proposition}
\begin{proof}
As shown in \cite{Dong2006}, multiplying the Dirac equation by $(t-r)^{-2\gamma} \del_t \Psi^* \gamma^0$ the proof follows from the differential identity
$$
\aligned
&\del_t \big( (t-r)^{-2\gamma} \Psi^* \Psi \big)
+
\del_a \big( (t-r)^{-2\gamma} \Psi^* \gamma^0 \gamma^a \Psi \big)
-
\del_t \big(  (t-r)^{-2\gamma} \big) \Psi^* \Psi
-
\del_a \big( (t-r)^{-2\gamma} \big) \Psi^* \gamma^0 \gamma^a \Psi
\\
&
= i\Psi^* \gamma^0 f - i f^* \gamma^0 \Psi,
\endaligned
$$
and the fact that
$$
\aligned
-
\del_t &\big(  (t-r)^{-2\gamma} \big) \Psi^* \Psi
-
\del_a \big( (t-r)^{-2\gamma} \big) \Psi^* \gamma^0 \gamma^a \Psi
\\
&=
\gamma (t-r)^{-2\gamma-1} \big( \Psi - (x_a/r) \gamma^0 \gamma^a \Psi \big)^* \big( \Psi - (x_a/r) \gamma^0 \gamma^a \Psi \big)
\geq 0.
\endaligned
$$

%and deriving the identity
%\bea\notag
%&\frac12 \del_t \Big[ (t-r)^{-\gamma} \big( (\del_t \phi)^2 + (\del_a \phi)^2 + m^2 \phi^2 \big) \Big] 
%- \sum_a \del_a \Big( (t-r)^{-\gamma} \del_t \phi \del_a \phi \Big)
%\\
%& + \frac{\gamma}{2}(t-r)^{-\gamma-1} \Big[\sum_a  \Big( \frac{x_a}{r} \del_t \phi + \del_a \phi\Big)^2 + m^2\phi^2 \Big] 
%= (t-r)^{-\gamma} f \del_t \phi.
%\eea
%The first term of the second line is non-negative. 
\end{proof}

\subsection{Estimates for null forms and commutators}
We next state a key estimate for null forms in terms of the hyperboloidal coordinates. The proof is standard and can be found in \cite[\textsection 4]{PLF-YM-book}.

\begin{lemma}\label{lem:null}
Let $\phi, \varphi$ be sufficiently regular functions with support in $\Kcal$ and define $Q_0(\phi, \varphi) \define \eta^{\alpha\beta}\del_\alpha \phi \del_\beta \varphi$. Then 
$$
\big|Q_0(\phi, \varphi)  \big|
\lesssim
\Big( \frac{s}{t} \Big)^2 \big| \del_t \phi  \cdot \del_t \varphi \big|
+
\sum_a \big( |\underline{\del}_a \phi  \cdot \del_t \varphi| + |\del_t \phi \cdot \underline{\del}_a \varphi| \big)
+
\sum_{a, b} \big| \underline{\del}_a \phi  \cdot \underline{\del}_b \varphi \big|.
$$
\end{lemma}

We also have the useful property that for the $Q_0$ null form:
$$\aligned
L_a Q_0( \phi,  \varphi) &= Q_0(L_a \phi,  \varphi) + Q_0( \phi, L_a \varphi), \quad
\del_\alpha Q_0( \phi,  \varphi) = Q_0(\del_\alpha \phi,  \varphi) + Q_0( \phi, \del_\alpha \varphi) .
\endaligned
$$

Besides the well-known commutation relations
$$
[\del_\alpha, -\Box + m^2 ]
=
[L_a, -\Box + m^2 ]
=
0,
\qquad
[i\gamma^\alpha \del_\alpha, \widehat{L}_a] = 0,
$$
valid for $m\geq 0$, we also need the following lemma to control some other commutators. A proof can be found in \cite[\textsection3]{PLF-YM-book} and \cite{PLF-YM-cmp}. 
\begin{lemma} \label{lem:commu}
Let $\Phi, \phi$ be a sufficiently regular $\CC^2$-valued (resp. $\RR$-valued) function supported in the region $\mathcal{K}$. Then, for any multi-indices $I$, there exist generic constants $C=C(|I|)>0$ such that
$$
\aligned
\big| [\del_\alpha, L_a] \Phi \big| + \big| [\del_\alpha, \widehat{L}_a] \Phi \big|
&\leq C \sum_\beta |\del_\beta  \Phi|,
\qquad 
\big| [L_a, L_b] \Phi  \big| + \big| [\widehat{L}_a, \widehat{L}_b] \Phi  \big|
\leq C \sum_c |L_c \Phi|,
\\
 \big| [Z^I, \del_\alpha] \phi \big| 
&\leq 
C \sum_{|J|<|I|} \sum_\beta \big|\del_\beta Z^J \phi \big|,
\\
 \big| [Z^I, \underdel_a] \phi \big| 
&\leq 
C \Big( \sum_{|J|<| I |} \sum_b \big|\underdel_b Z^J \phi \big| + t^{-1} \sum_{|J|\leq | I |} \big| Z^J\phi \big| \Big).
\endaligned
$$
Furthermore there exists a constant $C>0$ such that 
$$
\aligned
%\label{eq:est-cmt4}
\big| \del_\alpha (s/t) \big|
&\leq C s^{-1},
\qquad
%\label{eq:est-cmt5}
\big| L_a (s/t) \big| + \big| L_a L_b (s/t) \big|
&\leq C (s/t).
\endaligned
$$
Recall here that Greek indices $\alpha, \beta \in \{0,1,2\}$ and Roman indices $a,b \in \{1,2\}$. 
\end{lemma}

\subsection{Weighted Sobolev inequalities on hyperboloids}
%To conclude our preliminary section, w
We need certain weighted Sobolev inequalities  to obtain pointwise decay estimates for the Dirac field and the Klein-Gordon field.  

\begin{proposition}\label{prop-standard-Sobolev}
Let $\phi = \phi(t, x)$ be a sufficiently smooth function supported in the region $\Kcal$ and $\gamma \in \RR$. Then for all $s \geq 2$ we have
\begin{align}
\label{eq:Sobolev}
\sup_{\Hcal_s} \big| t \, \phi(t, x) \big|
&\leq C \sum_{|J| \leq 2} \big\| L^J \phi \big\|_{L^2_f(\Hcal_s)}.
\\
\label{eq:Sobolev2}
\sup_{\Hcal_s} \big| s  \, \phi(t, x) \big|
&\leq C \sum_{|J| \leq 2} \big\| (s/t)  L^J \phi \big\|_{L^2_f(\Hcal_s)},
\\
\label{eq:Sobolev2a}
\sup_{\Hcal_s} \big| s \, (t-r)^\gamma \, \phi(t, x) \big|
&\leq C \sum_{|J| \leq 2} \big\| (s/t) (t-r)^\gamma L^J \phi \big\|_{L^2_f(\Hcal_s)}.
\end{align}
\end{proposition}

We recall that such Sobolev inequalities involving hyperboloids were first introduced by Klainerman \cite{Klainerman85}, and then later appeared in work of H\"ormander \cite{Hormander}. In the above Proposition we have used the version given by H\"ormander in \cite{Hormander} where only the Lorentz boosts are required.
The estimate \eqref{eq:Sobolev2} follows by combining \eqref{eq:Sobolev} with the commutator estimates of Lemma \ref{lem:commu} and is more convenient to use for wave components.

We also have the following modified Sobolev inequalities for spinors which make use of the modified Lorentz boosts $\hL_a$. The proof follows from the fact that the difference between $L_a$ and $\widehat{L}_a$ is a constant matrix.

\begin{corollary}\label{corol-Dirac-Sobolev}
Let $\Psi = \Psi(t, x)$ be a sufficiently smooth $\CC^2$-valued  function supported in the region $\Kcal$. Then for all $s \geq 2$ we have
\bel{eq:Sobolev3}
\sup_{\Hcal_s} \big| t \, \Psi(t, x) \big|
\leq C \sum_{|J| \leq 2} \big\| \widehat{L}^J \Psi \big\|_{L^2_f(\Hcal_s)},
\ee
as well as
\bel{eq:Sobolev4}
\aligned
\sup_{\Hcal_s} \big| s \, \Psi(t, x) \big|
&\leq C \sum_{|J| \leq 2} \big\| (s/t) \widehat{L}^J \Psi \big\|_{L^2_f(\Hcal_s)},
\\
\sup_{\Hcal_s} \big| s \, (t-r)^\gamma \, \Psi(t, x) \big|
&\leq C \sum_{|J| \leq 2} \big\| (s/t) (t-r)^\gamma \widehat{L}^J \Psi \big\|_{L^2_f(\Hcal_s)}.
\endaligned
\ee
\end{corollary}

\subsection{Linear scattering}

To show linear scattering of the solution $(v, \psi)$ in the energy space in Theorems \ref{thm:main1} and \ref{thm:main2}, we need the following result, which gives a sufficient condition on linear scattering for Klein-Gordon and Dirac equations.

\begin{lemma}\label{lem:scattering}
Consider the Klein-Gordon equation
$$
-\Box u + u = F_u,
\qquad
(u, \del_t u)(t_0) = (u_0, u_1).
$$
If the source term satisfies
\bel{eq:scatter-KG}
\int_{t_0}^{+\infty} \| F_u \|_{L^2(\RR^2)} \, \di t
< +\infty,
\ee
then the solution $u$ scatters linearly in the energy space. That is,
there exists $u^+$, such that 
\be 
\lim_{t\to + \infty} \big(\|u - u^+\|_{L^2(\RR^2)} + \|\del (u-u^+)\|_{L^2(\RR^2)} \big) 
= 0,
\ee
in which $u^+$ is the solution to the free Klein-Gordon equation
$$
-\Box u^+ + u^+ = 0,
\qquad
(u^+, \del_t u^+)(t_0) = (u_0^+, u_1^+),
$$
for some $(u_0^+, u_1^+) \in H^1(\RR^2) \times L^2(\RR^2)$.

Similarly, consider the Dirac equation
$$
-i \gamma^\mu \del_\mu \Psi = F_{\Psi},
\qquad
\Psi (t_0) = \Psi_0.
$$
If the source term satisfies
\bel{eq:scatter-D}
\int_{t_0}^{+\infty} \| F_\Psi \|_{L^2(\RR^2)} \, \di t
< +\infty,
\ee
then the solution $\Psi$ scatters linearly in the energy space, i.e.,
there exists $\Psi^+$, such that 
\be 
\lim_{t\to + \infty} \|\Psi - \Psi^+\|_{L^2(\RR^2)} 
= 0,
\ee
in which $\Psi^+$ is the solution to the free Klein-Gordon equation
$$
-i \gamma^\mu \del_\mu \Psi^+ = 0,
\qquad
\Psi^+ (t_0) = \Psi_0^+,
$$
for some $\Psi^+_0 \in L^2(\RR^2)$.
\end{lemma}

The result in Lemma \ref{lem:scattering} is classical, and its proof can be found for instance in \cite{DongMa}. We note that the scattering result is valid on constant $t$ slices, while we work  on constant $s$ slices.

%==============================================================================================
\section{Hidden structure within the Dirac--Klein-Gordon equations}\label{sec:hidden}

\subsection{Transformations}\label{subsec:hidden-null}
In the present section we discuss three types of hidden structures which are present in the  Dirac--Klein-Gordon equations. These are in the spirit of Shatah's normal form method \cite{Shatah}. Identifying these structures plays an important role in our proof. \vspace{5pt}

\noindent\textbf{Type 1:} Consider a Klein-Gordon equation of the type $
(-\Box +1)v = w^2 + F_v,$
where $w$ satisfies an unspecified semilinear wave equation. 
If we set $\widetilde{v}
=
v - w^2,
$
then we have
$$
(-\Box +1)\widetilde{v}
=
F_v - 2 w(-\Box w )+ 2 Q_0(w, w).
$$
In particular, we can remove the wave-wave interaction $w^2$ at the expenses of bringing in cubic and null terms. This strategy of treating wave-wave interactions in Klein-Gordon equations was first introduced by Tsutsumi \cite{Tsutsumi} to study the Dirac-Proca equations in $\RR^{1+3}$.
\vspace{5pt}

\noindent\textbf{Type 2:}
Next we consider a wave equation with the form $
-\Box w = w v + F_w,
$
where $v$ satisfies an unspecified semilinear Klein-Gordon equation. 
If we set $
\widetilde{w} = w + w v,
$
then we have
$$
\aligned
- \Box \widetilde{w}
=
F_w + (-\Box w) v + w (-\Box v+v)- 2Q_0(w, v) .
\endaligned
$$
We can remove the  interaction term $wu$ at the expense of introducing null and cubic terms. 
\vspace{5pt}

\noindent\textbf{Type 3:}
In this final case, we consider a Dirac equation of the form $
-i\gamma^\mu \del_\mu \psi = v F\psi $
where $v$ satisfies an unspecified semilinear Klein-Gordon equation. If we set $\widetilde{\psi} = \psi + i \gamma^\nu \del_\nu (v F \psi)$ and use \eqref{eq:dirac-to-wave} then we find
%$$
%- i\gamma^\mu \del_\mu \widetilde{\psi}
%=
%- i\gamma^\mu \del_\mu \big(  \psi+i \gamma^\nu \del_\nu (v \psi)  \big)
%=
%i\gamma^\mu \del_\mu \psi 
%- \big(i\gamma^\mu \del_\mu \big) \big(  i \gamma^\nu \del_\nu (v \psi)  \big).
%$$
%By recalling the relation  we have
$$
- i\gamma^\mu \del_\mu \widetilde{\psi}
=
- i\gamma^\mu \del_\mu \psi  - \Box (v F \psi)
=
v F \psi +  (-\Box v)F\psi + v F (-\Box \psi) - 2 \eta^{\alpha\beta}\del_\alpha v F \del_\beta \psi.
$$
Thus we arrive at 
$$
-i\gamma^\mu \del_\mu \widetilde{\psi}
=
(-\Box v+v)F \psi + v F ( -\Box \psi ) - 2Q_0(v, F \psi).
$$
The nonlinear transformation has allowed us to cancel the Dirac-Klein-Gordon interaction $v \psi$ at the expense of introducing null and cubic terms. Such a transformation has, to the authors' knowledge, not been used before and is clearly inspired by the two prior transformations.

\subsection{Hidden Klein-Gordon structure in the Lorentz scalar ${\psi^*} \gamma^0 \psi$}

We now consider the Dirac--Dirac interaction term $
 \psi^* \gamma^0 \psi$
and show that it can be decomposed into terms with Klein-Gordon type factors. Roughly speaking, we call a field $\phi$ of `Klein-Gordon type' if its $L^2$ norm $ \| \phi \|_{L^2_f(\Hcal_s)}$
can be well controlled. Examples of `Klein-Gordon type' fields include 
$$
v, 
\quad
(s/t)\del_\alpha v,
\quad
(s/t) \psi,
\quad
\psi - (x^a/t)\gamma^0 \gamma^a \psi.
$$

\begin{definition}
Let $\Psi$ be a $\CC^2$-valued function. We define 
$$
(\Psi)_- \define \Psi - {x_a \over t} \gamma^0 \gamma^a \Psi,
\qquad
(\Psi)_+ \define \Psi + {x_a \over t} \gamma^0 \gamma^a \Psi \,.
$$
If no confusion arises, we use the abbreviation $\Psi_- = (\Psi)_-$.
\end{definition}

\begin{lemma}\label{lem:hidden-KG}
Let $\Psi, \Phi$ be two $\CC^2$-valued functions, then we have
$$
\Psi^* \gamma^0 \Phi
=
\tfrac14 \Big( (\Psi_-)^* \gamma^0 \Phi_-  + (\Psi_-)^* \gamma^0 \Phi_+ 
+ (\Psi_+)^* \gamma^0 \Phi_- + (s/t)^2 \Psi^* \gamma^0 \Phi  \Big).
$$
\end{lemma}

\begin{proof}

First we note $2 \Psi = \Psi_- + \Psi_+$ and $
2 \Phi = \Phi_- + \Phi_+$. Thus we have
$$
\aligned
4 \Psi^* \gamma^0 \Phi
&= 
\big( (\Psi_-)^* + (\Psi_+)^* \big)  \gamma^0 \big( \Phi_- + \Phi_+ \big)
\\
&=
(\Psi_-)^* \gamma^0 \Phi_-
+
(\Psi_-)^* \gamma^0 \Phi_+
+
(\Psi_+)^* \gamma^0 \Phi_-
+
(\Psi_+)^* \gamma^0 \Phi_+ \,.
\endaligned
$$
We expand the last term above, noting $(\gamma^0\gamma^a)^* = \gamma^0\gamma^a$, and find
$$
\aligned
(\Psi_+)^* \gamma^0 \Phi_+
&=
\big( \Psi^* + {x_a \over t}  \Psi^* \gamma^0 \gamma^a \big) \gamma^0 \big( \Phi + {x_b \over t} \gamma^0 \gamma^b \Phi \big)
\\
&=
\Psi^* \gamma^0 \Phi 
+
{x_a \over t}\Psi^* \gamma^0  \gamma^0 \gamma^a \Phi
+
{x_a \over t}  \Psi^* \gamma^0 \gamma^a \gamma^0 \Phi
+
{x_a \over t} {x_b \over t} \Psi^* \gamma^0 \gamma^a \gamma^0  \gamma^0 \gamma^b \Phi \,.
\endaligned
$$
Simple calculations give us
$$
\aligned
{x_a \over t} \Psi^* \gamma^0 \gamma^0 \gamma^a \Phi
+
{x_a \over t}  \Psi^* \gamma^0 \gamma^a \gamma^0 \Phi
=
0,
\endaligned
$$
and
\bel{eq:1001}
\aligned
{x_a \over t} {x_b \over t}  \Psi^* \gamma^0 \gamma^a \gamma^0  \gamma^0 \gamma^b \Phi
=
{x_a x_b \over t^2} \Psi^* \gamma^0 \gamma^a \gamma^b \Phi
=
- {r^2 \over t^2}  \Psi^* \gamma^0 \Phi \,.
\endaligned
\ee
Thus we are led to
\bel{eq:900}
\Psi_+^* \gamma^0 \Phi_+
=
\Psi^* \gamma^0 \Phi 
-
{r^2 \over t^2}  \Psi^* \gamma^0 \Phi
=
{s^2 \over t^2} \Psi^* \gamma^0 \Phi.
\ee
Gathering together the above results finishes the proof.
\end{proof}

\begin{rembold}\label{rem:hidden-KG}
The above Lemma gives the key improvement that the quadratic interaction term $\Psi^* \gamma^0 \Phi$ can be written in terms of other quadratic interactions which always involve at least one `Klein-Gordon type' field. 
It is also interesting to note that other Dirac-Dirac interactions terms do not possess the above useful decomposition. For example, replicating the argument for ${\psi^*}\psi$  in the proof of Lemma \ref{lem:hidden-KG}, we find \eqref{eq:1001} instead appears with a positive sign $+(r/t)^2  \Psi^* \gamma^0 \Phi$. This means that we cannot obtain a good factor of $(s/t)^2$ as in \eqref{eq:900}. Similar problems occur for ${\psi^*}\gamma^0\gamma^\mu\psi$. In this sense, general nonlinear terms $\psi^* H \psi$ under assumption \textbf{H1a} are more difficult to treat. 
\end{rembold}

Since the Dirac-Dirac interaction term $\psi^*\gamma^0\psi$ appears as a sourcing for the Klein-Gordon equation when \textbf{H2b} is assumed, we will need to act \emph{un}modified Lorentz boosts $L$ on this term. The following Lemma surprisingly shows that when distributing these Lorentz boosts across the interaction term, they in fact turn into the modified boosts $\hL$.

\begin{lemma}\label{lem:L-Dirac}
For any multi-index $|I|$ there exists a generic constant $C=C(|I|)>0$ such that 
$$
|Z^I(\psi^* \gamma^0\psi) |
\leq C
\sum_{|J|+|K|\leq |I|} |(\widehat{Z}^{J} \psi)^* \gamma^0 \widehat{Z}^{K}\psi|.
$$
\end{lemma}

\begin{proof}
Let $\Psi, \Phi$ be two $\CC^2$-valued functions. We will only consider the case with Lorentz boosts acting on the nonlinearity.
Since ${}^\ast$ denotes the conjugate transpose, and $(\gamma^0 \gamma^a)^* = (\gamma^0\gamma^a)$, we have the identity
$$
L_a (\Psi^*) = (\hL_a \Psi)^* + \frac12 \Psi^* (\gamma^0\gamma^a)^*
= (\hL_a \Psi)^* - \frac12 \Psi^* \gamma^a\gamma^0.
$$
and thus
\begin{align*}
L_a(\Psi^* \gamma^0 \Phi)
&= L_a(\Psi^*) \gamma^0 \Phi
+ \Psi^* \gamma^0 L_a(\Phi)
\\&=
\hL_a(\Psi^*) \gamma^0 \Phi - \frac12 \Psi^* \gamma^a \gamma^0 \gamma^0 \widetilde \Psi
+ \Psi^* \gamma^0 \hL_a(\Phi)
+ \frac12 \psi^* \gamma^0 \gamma^0 \gamma^a \widetilde \Psi
\\&=
\hL_a(\Psi^*) \gamma^0 \Phi 
+ \Psi^* \gamma^0 \hL_a(\Phi).
\end{align*}
Hence
$$
L_a(\bar{\psi}\psi) = L_a(\psi^*\gamma^0\psi)
=
(\hL_a \psi)^*\gamma^0\psi + \psi^*\gamma^0(\hL_a\psi).
$$
Similarly,
$$
L_b L_a(\bar{\psi}\psi)
= 
(\hL_b\hL_a\psi)^* \gamma^0 \psi + \psi^* \gamma^0 \hL_b \hL_a \psi 
+ (\hL_a\psi)^* \gamma^0 (\hL_b \psi) 
+ (\hL_b\psi)^* \gamma^0 (\hL_a \psi) .
$$
Carrying on gives the general pattern. 
\end{proof}

\subsection{Decay away from the light cone for differentiated Dirac components}

%The following important lemma crucially allows us to obtain extra $(t-r)$-decay for the $\del_t \psi$ component, which we use later when estimating null forms involving $\psi$. 
The following lemma is inspired by a similar result in the context of wave equations obtained in \cite[\textsection8.1, \textsection8.2]{PLF-YM-book}.
With the aid of Lemma \ref{lem:extra-t-r2}, we will be able to prove better estimates for the $\del \psi$ component; see for instance Propositions \ref{prop:psi-t} and \ref{prop:psi-t2}.

\begin{lemma}\label{lem:extra-t-r2}
Let $\Psi$ be a $\CC^2$-valued function solving  $
i\gamma^\mu \del_\mu \Psi = F_\Psi,
$
and supported in $\Kcal$. Then we have the following estimate
\bel{eq:del-psi2}
\big|\del_t \Psi \big|
\lesssim
{t\over t-r} \Big(\sum_a |\underline{\del}_a \Psi | + |F_\Psi| \Big).
\ee
\end{lemma}
\begin{proof}
We express the Dirac operator $i\gamma^\mu \del_\mu$ in the semi-hyperboloidal frame to get
$$
i \big( \gamma^0 - (x^a/t) \gamma^a \big) \del_t \Psi + i \gamma^a \underline{\del}_a \Psi = F_\Psi.
$$
Multiplying both sides by $\big( \gamma^0 - (x^b/t) \gamma^b \big)$ yields
$$
i \big( \gamma^0 - (x^b/t) \gamma^b \big) \big( \gamma^0 - (x^a/t) \gamma^a \big) \del_t \Psi + i \big( \gamma^0 - (x^b/t) \gamma^b \big) \gamma^a \underline{\del}_a \Psi = \big( \gamma^0 - (x^b/t) \gamma^b \big) F_\Psi.
$$
Simple calculations involving properties of the Dirac matrices imply
$$
\big( \gamma^0 - (x^b/t) \gamma^b \big) \big( \gamma^0 - (x^a/t) \gamma^a \big) 
= (s^2 /t^2).
$$
This leads us to
$$
i (s^2 /t^2) \del_t \Psi + i \big( \gamma^0 - (x^b/t) \gamma^b \big) \gamma^a \underline{\del}_a \Psi = \big( \gamma^0 - (x^b/t) \gamma^b \big) F_\Psi,
$$
which further implies
$$
\big| (s^2 /t^2) \del_t \Psi \big|
\leq
\big| \big( \gamma^0 - (x^b/t) \gamma^b \big) \gamma^a \underline{\del}_a \Psi \big|
+
\big| \big( \gamma^0 - (x^b/t) \gamma^b \big) F_\Psi \big|
\lesssim
\sum_a \big|  \underline{\del}_a \Psi \big|
+
\big| F_\Psi \big|.
$$
Finally we arrive at \eqref{eq:del-psi2} by recalling the following relations, which hold within the cone $\Kcal$,
$$
s^2 = t^2- r^2 = (t-r) (t+r),
\qquad
t\leq t+r \leq 2t.
$$
\end{proof}

%==============================================================================================
\section{Proof of Theorem \ref{thm:main1}}\label{sec:proof}

\subsection{Bootstrap assumptions and preliminary estimates}
Fix $N\in\mathbb{N}$ a large integer ($N \geq 7$ will end up working for our argument below). As shown by the local well-posedness theory in \cite[\textsection11]{PLF-YM-book}, initial data posed on the hypersurface $\{t_0=2\}$ and localised in the unit ball $\{x\in\RR^2:|x|\leq 1\}$ can be developed as a solution of \eqref{eq:D-KG} up to the initial hyperboloid $\{ s=s_0\}$ with the smallness \eqref{thm:data-assumpt} conserved. Thus there exists  $C_0>0$ such that the following bounds hold for all $|I|\leq N$:
\bea\label{eq:m1BApre}
\Ecal_1(s_0, Z^I v)^{1/2} + \Ecal^D (s_0, \widehat{Z}^I \psi)^{1/2} \leq C_0 \eps. 
\eea
Next, we assume that the following bounds hold for $s \in [s_0, s_1)$:  
\bel{eq:BA-Dirac}
\aligned
\Ecal^D (s, \widehat{Z}^I \psi)^{1/2}
+
\Ecal_1 (s, Z^I v)^{1/2}
&\leq C_1 \eps,
\quad
&|I|  &\leq N-2,  
\\
\Ecal^D (s, \widehat{Z}^I \psi)^{1/2}
+
\Ecal_1 (s, Z^I v)^{1/2}
&\leq C_1 \eps s^\delta,
\quad
&|I|  &= N-1,  
\\
\Ecal^D (s, \widehat{Z}^I \psi, 1)^{1/2}
+
s^{-1} \Ecal_1 (s, Z^I v)^{1/2}
&\leq C_1 \eps s^\delta,
\quad
&|I|  &= N.
\endaligned
\ee
In the above, the constant $C_1 \gg 1$ is to be determined, $\eps \ll 1$ measures the size of the initial data, and we let $C_1 \eps \ll 1$, and $0<\delta \leq \tfrac{1}{10}$. \underline{For the rest of section \ref{sec:proof}} we assume, without restating the fact, that \eqref{eq:BA-Dirac}  hold on a hyperboloidal time interval $[s_0, s_1)$  where
$$
s_1 \define \sup \{ s: s>s_0,\, \eqref{eq:BA-Dirac}\,\, \text{holds} \}.
$$
With the bounds in \eqref{eq:BA-Dirac}, we obtain the following preliminary $L^2$ and $L^\infty$ estimates.

\begin{proposition}\label{prop:L2}
For $s \in [s_0, s_1)$ we have
$$
\aligned
\big\| (s/t) \widehat{Z}^I \psi \big\|_{L^2_f(\Hcal_s)}
+
\big\| (s/t) Z^I \psi \big\|_{L^2_f(\Hcal_s)}
+
\big\| (\widehat{Z}^I \psi)_- \big\|_{L^2_f(\Hcal_s)}
&\lesssim
\left\lbrace
\begin{array}{ll}
C_1 \eps, & |I| \leq N-2,
\vspace{0.15cm}
\\
C_1 \eps s^\delta, & |I| \leq N-1,
\end{array}\right.
\\
\Big\| {(s/t) \widehat{Z}^I \psi \over (t-r) }\Big\|_{L^2_f(\Hcal_s)}
+
\Big\| {(s/t)  Z^I \psi \over (t-r)} \Big\|_{L^2_f(\Hcal_s)}
+
\Big\| { (\widehat{Z}^I \psi)_- \over (t-r)}\Big\|_{L^2_f(\Hcal_s)}
&\lesssim
C_1 \eps s^\delta,
\quad |I|\leq N,
\\
\big\| (s/t) \del Z^I v \big\|_{L^2_f(\Hcal_s)}
+
\big\| (s/t)  Z^I \del v \big\|_{L^2_f(\Hcal_s)}
+
\big\| Z^I v \big\|_{L^2_f(\Hcal_s)}
&\lesssim
\left\lbrace
\begin{array}{ll}
C_1 \eps, & |I| \leq N-2,
\vspace{0.15cm}
\\
C_1 \eps s^\delta, & |I| \leq N-1,
\vspace{0.15cm}
\\
C_1 \eps s^{1+\delta}, & |I| \leq N.
\end{array}\right.
\endaligned
$$

%$$
%\aligned
%\big\| (s/t) \widehat{Z}^I \psi \big\|_{L^2_f(\Hcal_s)}
%+
%\big\| (s/t) Z^I \psi \big\|_{L^2_f(\Hcal_s)}
%+
%\big\| (\widehat{Z}^I \psi)_- \big\|_{L^2_f(\Hcal_s)}
%&\lesssim
%C_1 \eps,
%&|I|  &\leq N-2,
%\\
%\big\| (s/t) \widehat{Z}^I \psi \big\|_{L^2_f(\Hcal_s)}
%+
%\big\| (s/t) Z^I \psi \big\|_{L^2_f(\Hcal_s)}
%+
%\big\| (\widehat{Z}^I \psi)_- \big\|_{L^2_f(\Hcal_s)}
%&\lesssim
%C_1 \eps s^\delta,
%&|I|  &\leq N-1,
%\\
%\Big\| {(s/t) \widehat{Z}^I \psi \over (t-r) }\Big\|_{L^2_f(\Hcal_s)}
%+
%\Big\| {(s/t)  Z^I \psi \over (t-r)} \Big\|_{L^2_f(\Hcal_s)}
%+
%\Big\| { (\widehat{Z}^I \psi)_- \over (t-r)}\Big\|_{L^2_f(\Hcal_s)}
%&\lesssim
%C_1 \eps s^\delta,
%&|I|  &\leq N,
%\\
%\big\| (s/t) \del Z^I v \big\|_{L^2_f(\Hcal_s)}
%+
%\big\| (s/t)  Z^I \del v \big\|_{L^2_f(\Hcal_s)}
%+
%\big\| Z^I v \big\|_{L^2_f(\Hcal_s)}
%&\lesssim
%C_1 \eps,
%&|I|  &\leq N-2,
%\\
%\big\| (s/t) \del Z^I v \big\|_{L^2_f(\Hcal_s)}
%+
%\big\| (s/t)  Z^I \del v \big\|_{L^2_f(\Hcal_s)}
%+
%\big\| Z^I v \big\|_{L^2_f(\Hcal_s)}
%&\lesssim
%C_1 \eps s^\delta,
%&|I|  &\leq N-1,
%\\
%\big\| (s/t) \del Z^I v \big\|_{L^2_f(\Hcal_s)}
%+
%\big\| (s/t)  Z^I \del v \big\|_{L^2_f(\Hcal_s)}
%+
%\big\| Z^I v \big\|_{L^2_f(\Hcal_s)}
%&\lesssim
%C_1 \eps s^{1+\delta},
%&|I|  &\leq N.
%\endaligned
%$$

\end{proposition}
\begin{proof}
The estimates for $\psi$ follow from the definition of the energy functional $\Ecal^D(s, \psi), \Ecal^D(s, \psi, 1)$, the decomposition \eqref{eq:D-fctnal-2}, the commutator estimates in Lemma \ref{lem:commu} and the fact that the difference between $L_a$ and $\widehat{L}_a$ is a constant matrix. The estimates for the Klein-Gordon field follow from the definition of the energy functional $\Ecal_1(s, v)$ and the commutator estimates in Lemma \ref{lem:commu}. 
\end{proof}

Next we  derive the following pointwise estimates.

\begin{proposition}\label{prop:Linfty}
For $s \in [s_0, s_1)$ we have
$$
\aligned
\big| \widehat{Z}^I \psi \big| 
+
\big| Z^I \psi \big| 
+
(t/s)\big| (\widehat{Z}^I \psi)_- \big| 
&\lesssim
\left\lbrace
\begin{array}{ll}
C_1 \eps s^{-1}, & |I| \leq N-4,
\vspace{0.15cm}
\\
C_1 \eps s^{-1+\delta}, & |I|  \leq N-3,
\end{array}
\right.
\\
\big| \del Z^I v \big|
+
\big| Z^I \del v \big|
+
(t/s) \big| Z^I v \big|
&\lesssim
\left\lbrace
\begin{array}{ll}
C_1 \eps s^{-1}, & |I| \leq N-4,
\vspace{0.15cm}
\\
C_1 \eps s^{-1+\delta}, & |I|  \leq N-3.
\end{array}\right.
\endaligned
$$
%$$
%\aligned
%\big| \widehat{Z}^I \psi \big| 
%+
%\big| Z^I \psi \big| 
%+
%(t/s)\big| (\widehat{Z}^I \psi)_- \big| 
%&\lesssim
%C_1 \eps s^{-1},
%\qquad
%&|I| \leq N-4,
%\\
%\big| \widehat{Z}^I \psi \big| 
%+
%\big| Z^I \psi \big| 
%+
%(t/s) \big| (\widehat{Z}^I \psi)_- \big| 
%&\lesssim
%C_1 \eps s^{-1+\delta},
%\qquad
%&|I|  \leq N-3,
%\\
%\big| \del Z^I v \big|
%+
%\big| Z^I \del v \big|
%+
%(t/s)\big| Z^I v \big|
%&\lesssim
%C_1 \eps s^{-1},
%\qquad
%&|I|  \leq N-4,
%\\
%\big| \del Z^I v \big|
%+
%\big| Z^I \del v \big|
%+
%(t/s) \big| Z^I v \big|
%&\lesssim
%C_1 \eps s^{-1+\delta},
%\qquad
%&|I| \leq N-3.
%\endaligned
%$$
\end{proposition}

\begin{proof}
To show the estimates for the Klein-Gordon components $v$ and $\del v$ we combine the estimates from Proposition \ref{prop:L2} with the Sobolev estimates from Proposition \ref{prop-standard-Sobolev}. 
To prove the estimates for $\widehat{Z}^I \psi$, and thus $Z^I \psi$, we combine Proposition \ref{prop:L2}  with the Dirac-type Sobolev estimates from Corollary \ref{corol-Dirac-Sobolev}. Finally to prove the estimates for $(\psi)_-$ and derivatives thereof, we note $\gamma^0\gamma^0=I_2$ in order to show the commutator identity
$$
[\hatL_b,\gamma^0-(x^a/t)\gamma^a]\psi 
= -(x^b/t)(\gamma^0-(x^a/t)\gamma^a)\psi
= -(x^b/t) \gamma^0 (\psi)_-\,.
$$
This implies
$$
\aligned
{[}\hatL_b,I_2-(x^a/t)\gamma^0\gamma^a ] \psi 
&=[\hatL_b,\gamma^0(\gamma^0-(x^a/t)\gamma^a)]\psi 
\\
&= [\hatL_b,\gamma^0]\gamma^0\psi_- + \gamma^0 [\hatL_b,\gamma^0-(x^a/t)\gamma^a]\psi 
\\
&= -\big( \gamma^0\gamma^b+(x^b/t)\big)(\psi)_- \,.
\endaligned
$$
We can control this error term since $|x^b/t| \leq 1$ in the cone. 
Using these calculations, we can compute
$$
\aligned	
{[}\hatL_c\hatL_b,I_2-(x^a/t)\gamma^0\gamma^a ]\psi
&= -(\gamma^0\gamma^c+(x^c/t))(\hatL_b\psi)_- -(\gamma^0\gamma^b+(x^b/t))(\hatL_c\psi)_-
\\
&\quad +\big[ (x^b/t) \gamma^0 \gamma^c + (x^c/t)\gamma^0 \gamma^b + 2 (x^cx^b)/t^2\big] \psi_-.
\endaligned
$$
Thus, using the first Sobolev estimate in Corollary \ref{corol-Dirac-Sobolev},
$$
\sup_{\Hcal_s} |t \psi_-| \lesssim \sum_{|J|\leq 2} \| \hatL^J \psi_-\|_{L^2_f(\Hcal_s)}
= \sum_{|J|\leq 2} \| \hatL^J (I_2-(x^a/t)\gamma^0\gamma^a)\psi\|_{L^2_f(\Hcal_s)}
\lesssim \sum_{|J|\leq 2} \| (\hatL^J \psi)_-\|_{L^2_f(\Hcal_s)}.
$$
The estimates for $(\widehat{Z}^I \psi)_-$ follow in the same way and the proof is complete. 
\end{proof}

\begin{proposition}\label{prop:psi-t}
The following weighted $L^2$-estimates are valid for $s \in [s_0, s_1)$
$$
\aligned
\big\| (t-r) (s/t) \del Z^I \psi \big\|_{L^2_f(\Hcal_s)}
+
\big\| (t-r) (s/t) \del \widehat{Z}^I \psi \big\|_{L^2_f(\Hcal_s)}
&\lesssim
C_1 \eps s^{\delta},
\quad |I|  \leq N-2,
\\
\big\| (s/t) \del Z^I \psi \big\|_{L^2_f(\Hcal_s)}
+
\big\| (s/t) \del \widehat{Z}^I \psi \big\|_{L^2_f(\Hcal_s)}
&\lesssim
C_1 \eps s^{\delta},
\quad |I| \leq N-1,
\endaligned
$$
and the following pointwise estimates also hold for $s \in [s_0, s_1)$
$$
\aligned
\big| \del Z^I \psi \big|
+
\big| \del \widehat{Z}^I \psi \big|
&\lesssim
C_1 \eps (t-r)^{-1} s^{-1+\delta},
\quad |I|  \leq N-4.
\endaligned
$$
\end{proposition}
\begin{proof}
We first act $ \widehat{Z}^I$, with $|I| \leq N-4$, to the $\psi$ equation in \eqref{eq:D-KG} to find
$$
-i \gamma^\mu \del_\mu \widehat{Z}^I \psi 
= \widehat{Z}^I \big(v \psi\big).
$$
Then by Lemma \ref{lem:extra-t-r2} we obtain
$$
\aligned
\big| \del_t \widehat{Z}^I \psi \big|
%&\lesssim
%{t\over t-r} \Big(\sum_a \big|\underline{\del}_a \widehat{Z}^I \psi \big| + \big|\widehat{Z}^I (v \psi)\big| \Big)
\lesssim
{t\over t-r} \Big(t^{-1} \sum_a \big|L_a \widehat{Z}^I \psi \big| + \big|\widehat{Z}^I (v \psi)\big| \Big)
\lesssim
C_1 \eps (t-r)^{-1} s^{-1+\delta},
\endaligned
$$
in which we used the pointwise decay results of Proposition \ref{prop:Linfty}. The estimates $\big| \del_t Z^I \psi \big|$ are a simple consequence of the above, while the case $\big| \del_a \widehat{Z}^I \psi \big|$ (with $a=1, 2$) can be seen from the relation
$$
\del_a \widehat{Z}^I \psi 
= -{x_a \over t} \del_t \widehat{Z}^I \psi + \underline{\del}_a \widehat{Z}^I \psi.
$$
Finally the $L^2$--type estimates follow in a similar way, by combining Lemma \ref{lem:extra-t-r2}  with Propositions \ref{prop:L2} and \ref{prop:Linfty}. 
\end{proof}

\subsection{Nonlinear transformations and corresponding estimates}\label{subsec:nonlinearTrans}

Next, we introduce nonlinear transformations in the spirit of Shatah's normal form method \cite{Shatah}. These are key to closing the low order bootstraps.
\begin{lemma}\label{lem:transf-v}
Let $\widetilde{v} \define v - \psi^* H \psi$. Then $\widetilde{v}$ solves the following Klein-Gordon equation
\bel{eq:KG-new}
-\Box \widetilde{v} + \widetilde{v}
= -i \del_\nu(v\psi^*) (H \gamma^\nu)^* \psi +  i \psi^* H \gamma^\nu \del_\nu (v\psi) + 2 Q_0( \psi, H \psi). 
%\eta^{\alpha\beta}\del_\alpha \psi^* H \del_\beta \psi.
\ee
\end{lemma}
\begin{proof}
This nonlinear transformation was introduced in \cite{Tsutsumi}. The required result follows by using \eqref{eq:D-KG} to deduce
\bel{eq:Wave-Dirac01}
\aligned
-\Box \psi
=
i\gamma^\nu \del_\nu \big( -i\gamma^\mu \del_\mu \psi \big)
=
i\gamma^\nu \del_\nu (v \psi).
\endaligned
\ee
%, and we revisit the proof. We act the Klein-Gordon operator to $\widetilde{v}$ to obtain
%$$
%\aligned
%-\Box \widetilde{v} + \widetilde{v}
%=
%-\Box \big(v - \psi^* H \psi\big) + \big(v - \psi^* H \psi\big)
%=
%-\Box v + v - \psi^* H \psi + \Box \big( \psi^* H \psi \big).
%\endaligned
%$$
%Finally recalling the original equations in \eqref{eq:D-KG} leads us to \eqref{eq:KG-new}.
\end{proof}

\begin{lemma}\label{lem:NL-KGt}
We have
\be\notag
\aligned
\big\| Z^I \big(\psi^* H \gamma^\nu \del_\nu (v\psi) \big) \big\|_{L^2_f(\Hcal_\tau)} 
&\lesssim
\left\lbrace
\begin{array}{ll}
(C_1 \eps)^3 \tau^{-2+2\delta}, & |I| \leq N-2,
\vspace{0.15cm}
\\
(C_1 \eps)^3 \tau^{-1+\delta}, & |I| \leq N-1,
\end{array} \right.
\\
\big\| Z^I \big(\del_\alpha \psi^* H \del^\alpha \psi \big) \big\|_{L^2_f(\Hcal_\tau)} 
&\lesssim
(C_1 \eps)^2 \tau^{-2+2\delta},
\qquad
|I| \leq N-1.
\endaligned
\ee
%
%\be
%\aligned
%\big\| Z^I \big(\psi^* H \gamma^\nu \del_\nu (v\psi) \big) \big\|_{L^2_f(\Hcal_\tau)} 
%&\lesssim
%(C_1 \eps)^3 \tau^{-2+2\delta},
%\qquad
%&|I| \leq N-2,
%\\
% \big\| Z^I \big(\psi^* H \gamma^\nu \del_\nu (v\psi) \big) \big\|_{L^2_f(\Hcal_\tau)} 
%&\lesssim
%(C_1 \eps)^3 \tau^{-1+\delta},
%\qquad
%&|I| \leq N-1,
%\\
%\big\| Z^I \big(\del_\alpha \psi^* H \del^\alpha \psi \big) \big\|_{L^2_f(\Hcal_\tau)} 
%&\lesssim
%(C_1 \eps)^2 \tau^{-2+2\delta},
%\qquad
%&|I| \leq N-1.
%\endaligned
%\ee
\end{lemma}
\begin{proof}
We estimate each of these three quantities in turn. 

\noindent\textbf{Step 1:} Estimate of $\big\| Z^I \big(\psi^* H \gamma^\nu \del_\nu (v\psi) \big) \big\|_{L^2_f(\Hcal_\tau)} $ with $|I|\leq N-2$.
We first decompose the term into three pieces
$$
\aligned
\big\| Z^I \big(\psi^* &H \gamma^\nu \del_\nu (v\psi) \big) \big\|_{L^2_f(\Hcal_\tau)}
\\
\lesssim
&\sum_{\substack{|I_1| + |I_3|\leq N-3 \\ |I_2|\leq |I|}} \big\| |Z^{I_1} \psi| |Z^{I_2} \del v| | Z^{I_3} \psi| \big\|_{L^2_f(\Hcal_\tau)}
+
\sum_{\substack{|I_1| + |I_2|\leq N-3 \\ |I_3|\leq |I|}} \big\| |Z^{I_1} \psi| |Z^{I_2} v| | Z^{I_3} \del \psi| \big\|_{L^2_f(\Hcal_\tau)}
\\
+
&\sum_{\substack{|I_2| + |I_3|\leq N-3 \\ |I_1|\leq |I|}} \big\| |Z^{I_1} \psi| |Z^{I_2} v| | Z^{I_3} \psi| \big\|_{L^2_f(\Hcal_\tau)}
=: \mathcal{A}_{1a} + \mathcal{A}_{1b} + \mathcal{A}_{1c}.
\endaligned
$$
We now bound
$$
\aligned
\mathcal{A}_{1a}
\lesssim
&\sum_{\substack{|I_1| + |I_3|\leq N-3 \\ |I_2|\leq |I|}} \big\| Z^{I_1} \psi\big\|_{L^\infty(\Hcal_\tau)} \big\|Z^{I_2} \del v\big\|_{L^2_f(\Hcal_\tau)} \big\| Z^{I_3} \psi \big\|_{L^\infty(\Hcal_\tau)}
\\
\lesssim
&\sum_{\substack{|I_1| + |I_3|\leq N-3 \\ |J|\leq |I|+1}} \big\| Z^{I_1} \psi\big\|_{L^\infty(\Hcal_\tau)} \big\|Z^J v\big\|_{L^2_f(\Hcal_\tau)} \big\| Z^{I_3} \psi \big\|_{L^\infty(\Hcal_\tau)}
\lesssim
(C_1 \eps)^3 \tau^{-2+2\delta},
\endaligned
$$
in which we used Lemma \ref{lem:commu}, Proposition \ref{prop:L2}, and Proposition \ref{prop:Linfty}.
We continue to estimate
$$
\aligned
\mathcal{A}_{1b}
\lesssim
&\sum_{\substack{|I_1| + |I_2|\leq N-3 \\ |I_3|\leq |I|}} \big\| Z^{I_1} \psi\big\|_{L^\infty(\Hcal_\tau)} \big\|(t/\tau) Z^{I_2} v\big\|_{L^\infty(\Hcal_\tau)} \big\| (\tau/t) Z^{I_3} \del \psi \big\|_{L^2_f(\Hcal_\tau)}
\\
\lesssim
&\sum_{\substack{|I_1| + |I_2|\leq N-3 \\ |J|\leq |I|+1}} \big\| Z^{I_1} \psi\big\|_{L^\infty(\Hcal_\tau)} \big\|(t/\tau) Z^{I_2} v\big\|_{L^\infty(\Hcal_\tau)} \big\| (\tau/t) Z^{J} \psi \big\|_{L^2_f(\Hcal_\tau)}
\\
\lesssim
&(C_1 \eps)^3 \tau^{-2+2\delta},
\endaligned
$$
in which we used Lemma \ref{lem:commu}, Proposition \ref{prop:L2}, and Proposition \ref{prop:Linfty}.
We then get
$$
\aligned
\mathcal{A}_{1c}
\lesssim
\sum_{\substack{|I_2| + |I_3|\leq N-3 \\ |I_1|\leq |I|}} \big\| (\tau/t) Z^{I_1} \psi\big\|_{L^2_f(\Hcal_\tau)} \big\|(t/\tau) Z^{I_2} v\big\|_{L^\infty(\Hcal_\tau)} \big\| Z^{I_3} \psi \big\|_{L^\infty(\Hcal_\tau)}
\lesssim
(C_1 \eps)^3 \tau^{-2+\delta},
\endaligned
$$
in which we used Propositions \ref{prop:L2} and \ref{prop:Linfty}.

Thus we obtain
\be \notag
\big\| Z^I \big(\psi^* H \gamma^\nu \del_\nu (v\psi) \big) \big\|_{L^2_f(\Hcal_\tau)} 
\lesssim
(C_1 \eps)^3 \tau^{-2+2\delta},
\qquad
|I|\leq N-2.
\ee

\noindent\textbf{Step 2:} Estimate of $\big\| Z^I \big(\psi^* H \gamma^\nu \del_\nu (v\psi) \big) \big\|_{L^2_f(\Hcal_\tau)} $ with $|I|\leq N-1$.
We decompose the term into three pieces
$$
\aligned
\big\| Z^I \big(\psi^*& H \gamma^\nu \del_\nu (v\psi) \big) \big\|_{L^2_f(\Hcal_\tau)}
\\
\lesssim
&\sum_{\substack{|I_1| + |I_3|\leq N-4 \\ |I_2|\leq |I|}} \big\| |Z^{I_1} \psi| |Z^{I_2} \del v| | Z^{I_3} \psi| \big\|_{L^2_f(\Hcal_\tau)}
+
\sum_{\substack{|I_1| + |I_2|\leq N-3 \\ |I_3|\leq |I|}} \big\| |Z^{I_1} \psi| |Z^{I_2} v| | Z^{I_3} \del \psi| \big\|_{L^2_f(\Hcal_\tau)}
\\
+
&\sum_{\substack{|I_2| + |I_3|\leq N-3 \\ |I_1|\leq |I|}} \big\| |Z^{I_1} \psi| |Z^{I_2} v| | Z^{I_3} \psi| \big\|_{L^2_f(\Hcal_\tau)}
=: \mathcal{A}_{2a} + \mathcal{A}_{2b} + \mathcal{A}_{2c}.
\endaligned
$$
We first estimate
$$
\aligned
\mathcal{A}_{2a}
\lesssim
&\sum_{\substack{|I_1| + |I_3|\leq N-4 \\ |I_2|\leq |I|}} \big\| Z^{I_1} \psi\big\|_{L^\infty(\Hcal_\tau)} \big\|Z^{I_2} \del v\big\|_{L^2_f(\Hcal_\tau)} \big\| Z^{I_3} \psi \big\|_{L^\infty(\Hcal_\tau)}
\\
\lesssim
&\sum_{\substack{|I_1| + |I_3|\leq N-4 \\ |J|\leq |I|+1}} \big\| Z^{I_1} \psi\big\|_{L^\infty(\Hcal_\tau)} \big\|Z^J v\big\|_{L^2_f(\Hcal_\tau)} \big\| Z^{I_3} \psi \big\|_{L^\infty(\Hcal_\tau)}
\lesssim
(C_1 \eps)^3 \tau^{-1+\delta},
\endaligned
$$
in which we used Lemma \ref{lem:commu}, Proposition \ref{prop:L2}, and Proposition \ref{prop:Linfty}.
We now bound
$$
\aligned
\mathcal{A}_{2b}
\lesssim
&\sum_{\substack{|I_1| + |I_2|\leq N-3 \\ |I_3|\leq |I|}} \big\| Z^{I_1} \psi\big\|_{L^\infty(\Hcal_\tau)} \big\|(t/\tau) Z^{I_2} v\big\|_{L^\infty(\Hcal_\tau)} \big\| (\tau/t) Z^{I_3} \del \psi \big\|_{L^2_f(\Hcal_\tau)}
\\
\lesssim
&\sum_{\substack{|I_1| + |I_2|\leq N-3 \\ |J|\leq |I|}} \big\| Z^{I_1} \psi\big\|_{L^\infty(\Hcal_\tau)} \big\|(t/\tau) Z^{I_2} v\big\|_{L^\infty(\Hcal_\tau)} \big\| (\tau/t) \del Z^{J} \psi \big\|_{L^2_f(\Hcal_\tau)}
\\
\lesssim
&(C_1 \eps)^3 \tau^{-2+2\delta},
\endaligned
$$
in which we used Lemma \ref{lem:commu}, Proposition \ref{prop:Linfty}, and Proposition \ref{prop:psi-t}.
We then obtain
$$
\aligned
\mathcal{A}_{2c}
\lesssim
\sum_{\substack{|I_2| + |I_3|\leq N-3 \\ |I_1|\leq |I|}} \big\| (\tau/t) Z^{I_1} \psi\big\|_{L^2_f(\Hcal_\tau)} \big\|(t/\tau) Z^{I_2} v\big\|_{L^\infty(\Hcal_\tau)} \big\| Z^{I_3} \psi| \big\|_{L^\infty(\Hcal_\tau)}
\lesssim
(C_1 \eps)^3 \tau^{-2+2\delta},
\endaligned
$$
in which we used Proposition \ref{prop:L2} and Proposition \ref{prop:Linfty}.

In conclusion, we get
\be \notag
\big\| Z^I \big(\psi^* H \gamma^\nu \del_\nu (v\psi) \big) \big\|_{L^2_f(\Hcal_\tau)} 
\lesssim
(C_1 \eps)^3 \tau^{-1+\delta},
\qquad
|I|\leq N-1.
\ee

\noindent\textbf{Step 3:} Estimate of $ \big\| Z^I \big(\del_\alpha \psi^* H \del^\alpha \psi\big) \big\|_{L^2_f(\Hcal_\tau)} $ with $|I|\leq N-1$.

First, according to Lemma \ref{lem:commu} and Lemma \ref{lem:null} we have
$$
\aligned
\big\| Z^I \big(\del_\alpha \psi^* &H \del^\alpha \psi\big) \big\|_{L^2_f(\Hcal_\tau)}
\\
\lesssim
&\sum_{|I_1|+|I_2|\leq |I|} \big\| (\tau/t)^2 |\del_t Z^{I_1} \psi| |\del_t Z^{I_2} \psi| \big\|_{L^2_f(\Hcal_\tau)}
+
\sum_{|I_1|+|I_2|\leq |I|, a} \big\| |\underline{\del}_a Z^{I_1} \psi| |\del_t Z^{I_2} \psi| \big\|_{L^2_f(\Hcal_\tau)}
\\
+
&\sum_{|I_1|+|I_2|\leq |I|, a, b} \big\| |\underline{\del}_a Z^{I_1} \psi| |\underline{\del}_b Z^{I_2} \psi| \big\|_{L^2_f(\Hcal_\tau)}
=: \mathcal{A}_{3a} + \mathcal{A}_{3b} + \mathcal{A}_{3c}.
\endaligned
$$
We next estimate
$$
\aligned
\mathcal{A}_{3a}
\lesssim
\sum_{|I_1|\leq |I|,\, |I_2|\leq N-4} \big\| (\tau/t) \del_t Z^{I_1} \psi\big\|_{L^2_f(\Hcal_\tau)} \big\|(\tau/t) \del_t Z^{I_2} \psi| \big\|_{L^\infty(\Hcal_\tau)}
\lesssim
(C_1 \eps)^2 \tau^{-2+2\delta},
\endaligned
$$
in which we used Proposition \ref{prop:psi-t}.
We then bound
\begin{align*}
\mathcal{A}_{3b}
\lesssim
&\sum_{\substack{|I_1|\leq |I|\\ |I_2|\leq N-4, \, a}} \big\| (\tau/t) (t-r)^{-1} L_a Z^{I_1} \psi\big\|_{L^2_f(\Hcal_\tau)} \big\|\tau^{-1} (t-r) \del_t Z^{I_2} \psi| \big\|_{L^\infty(\Hcal_\tau)}
\\
&\quad +
\sum_{\substack{|I_1|\leq N-4\\ |I_2|\leq |I|, \, a}}  \big\| \tau^{-1} L_a Z^{I_1} \psi\big\|_{L^\infty(\Hcal_\tau)} \big\|(\tau/t) \del_t Z^{I_2} \psi \big\|_{L^2_f(\Hcal_\tau)}
\\
\lesssim
&\sum_{\substack{|J|\leq |I|+1\\ |I_2|\leq N-3, \, a}} \big\| (\tau/t) (t-r)^{-1}  Z^{J} \psi\big\|_{L^2_f(\Hcal_\tau)} \big\|\tau^{-1} (t-r) \del_t Z^{I_2} \psi| \big\|_{L^\infty(\Hcal_\tau)}
\\
& \quad + \sum_{\substack{|J|\leq N-3\\ |I_2|\leq |I|, \, a}}  \big\| \tau^{-1}  Z^{J} \psi\big\|_{L^\infty(\Hcal_\tau)} \big\|(\tau/t) \del_t Z^{I_2} \psi \big\|_{L^2_f(\Hcal_\tau)}
\\
\lesssim
&(C_1 \eps)^2 \tau^{-2+2\delta},
\end{align*}
in which we used Lemma \ref{lem:commu}, Proposition \ref{prop:L2}, and Proposition \ref{prop:psi-t}.
Easily we can show
$$
\mathcal{A}_{3c}
\lesssim
(C_1 \eps)^2 \tau^{-2}.
$$

To sum up, we get
\be \notag
\big\| Z^I \big(\del_\alpha \psi^* H \del^\alpha \psi\big) \big\|_{L^2_f(\Hcal_\tau)}
\lesssim
(C_1 \eps)^2 \tau^{-2+2\delta},
\qquad
|I|\leq N-1.
\ee

\end{proof}

Next we introduce a nonlinear transformation of Type 3 as discussed in Section \ref{subsec:hidden-null}. 
\begin{lemma}\label{lemma:EoM-tilde-psi}
Let $\widetilde{\psi} \define \psi + i\gamma^\nu \del_\nu (v \psi)$. Then $\widetilde{\psi}$ solves the following Dirac equation
\bel{eq:Dirac-new}
-i\gamma^\mu \del_\mu \widetilde{\psi}
=
\big(\psi^* H \psi\big) \psi + i \gamma^\nu v \del_\nu (v \psi) - 2 \del_\alpha v \del^\alpha \psi.
\ee
\end{lemma}
\begin{proof}
A straightforward application of \eqref{eq:D-KG}, \eqref{eq:dirac-to-wave} and \eqref{eq:Wave-Dirac01} yields the desired result. 
%We act the Dirac operator on $\widetilde{\psi}$ to get
%$$
%\aligned
%-i\gamma^\mu \del_\mu \widetilde{\psi}
%&=
%-i\gamma^\mu \del_\mu \big(  \psi + i\gamma^\nu \del_\nu (v \psi) \big)
%=
%-i\gamma^\mu \del_\mu \psi - \Box (v \psi)
%\\
%&=
%-i\gamma^\mu \del_\mu \psi + (-\Box v) \psi + v (-\Box \psi) - 2 \eta^{\alpha\beta}\del_\alpha v \del_\beta \psi.
%\endaligned
%$$
%We recall
%$$
%\aligned
%-\Box \psi
%=
%i\gamma^\nu \del_\nu \big( -i\gamma^\mu \del_\mu \psi \big)
%=
%i\gamma^\nu \del_\nu (v \psi).
%\endaligned
%$$
%Thus we arrive at the desired result by using the original equations in \eqref{eq:D-KG}.
\end{proof}

\begin{lemma}\label{lem:NL-Dt}
We have
\be \notag
\aligned
\big\| \widehat{Z}^I \big( (\psi^* H \psi) \psi  \big) \big\|_{L^2_f(\Hcal_\tau)} 
&\lesssim (C_1 \eps)^3 \tau^{-1+\delta}, 
\qquad
|I| \leq N-1,
\\
\big\| \widehat{Z}^I \big( \gamma^\nu v \del_\nu (v \psi) \big) \big\|_{L^2_f(\Hcal_\tau)} 
&\lesssim (C_1 \eps)^3 \tau^{-2+2\delta}, 
\qquad
|I| \leq N-1,
\\
\big\| \widehat{Z}^I \big( \del_\alpha v \del^\alpha \psi \big) \big\|_{L^2_f(\Hcal_\tau)} 
&\lesssim 
\left\lbrace
\begin{array}{ll}
(C_1 \eps)^2 \tau^{-2+2\delta}, & |I| \leq N-2,
\vspace{0.15cm}
\\
(C_1 \eps)^2 \tau^{-1+\delta}, & |I| \leq N-1.
\end{array}
\right.
\endaligned
\ee

%\be 
%\aligned
%\big\| \widehat{Z}^I \big( (\psi^* H \psi) \psi  \big) \big\|_{L^2_f(\Hcal_\tau)} 
%\lesssim (C_1 \eps)^3 s^{-1+\delta}, 
%\qquad
%|I| \leq N-1,
%\\
%\big\| \widehat{Z}^I \big( \gamma^\nu v \del_\nu (v \psi) \big) \big\|_{L^2_f(\Hcal_\tau)} 
%\lesssim (C_1 \eps)^3 \tau^{-2+2\delta}, 
%\qquad
%|I| \leq N-1,
%\\
%\big\| \widehat{Z}^I \big( \del_\alpha v \del^\alpha \psi \big) \big\|_{L^2_f(\Hcal_\tau)} 
%\lesssim (C_1 \eps)^2 \tau^{-2+2\delta}, 
%\qquad
%|I| \leq N-2,
%\\
%\big\| \widehat{Z}^I \big( \del_\alpha v \del^\alpha \psi \big) \big\|_{L^2_f(\Hcal_\tau)} 
%\lesssim (C_1 \eps)^2 s^{-1+\delta}, 
%\qquad
%|I| \leq N-1.
%\endaligned
%\ee
\end{lemma}
\begin{proof}
We bound the terms one by one.

\noindent\textbf{Step 1:} We start by estimating $\big\| \widehat{Z}^I \big( (\psi^* H \psi) \psi \big) \big\|_{L^2_f(\Hcal_\tau)} $ for $|I|\leq N-1$.
We find that
$$
\aligned
\big\| \widehat{Z}^I &\big( (\psi^* H \psi) \psi  \big) \big\|_{L^2_f(\Hcal_\tau)} 
\\
&\lesssim
\sum_{|I_1|+|I_2|+|I_3|\leq |I|} \big\|  (Z^{I_1}\psi)^* H (Z^{I_2} \psi) \widehat{Z}^{I_3} \psi \big\|_{L^2_f(\Hcal_\tau)} 
\\
&\lesssim
\sum_{\substack{|I_1|+|I_2|\leq N-4\\|I_3|\leq |I|}} \big\|  |Z^{I_1}\psi|  |Z^{I_2} \psi| |\widehat{Z}^{I_3} \psi| \big\|_{L^2_f(\Hcal_\tau)} 
+
\sum_{\substack{|I_1|\leq |I|\\ |I_2|+|I_3|\leq N-4}} \big\|  |Z^{I_1}\psi|  |Z^{I_2} \psi| |\widehat{Z}^{I_3} \psi| \big\|_{L^2_f(\Hcal_\tau)} 
\\
& =: 
 \mathcal{B}_{1a} + \mathcal{B}_{1b}.
\endaligned
$$
Easily, we get
$$
\aligned
\mathcal{B}_{1a}
\lesssim
\sum_{\substack{|I_1|+|I_2|\leq N-4\\ |I_3|\leq |I|}} \big\| (t/\tau) |Z^{I_1}\psi|  |Z^{I_2} \psi|\big\|_{L^\infty(\Hcal_\tau)}  \big\| (\tau/t) \widehat{Z}^{I_3} \psi \big\|_{L^2_f(\Hcal_\tau)} 
\lesssim
 (C_1 \eps)^3 \tau^{-1+\delta},
\endaligned
$$
in which we used Propositions \ref{prop:L2} and \ref{prop:Linfty}.
In the same way, we have
$$
\mathcal{B}_{1b}
\lesssim
 (C_1 \eps)^3 \tau^{-1+\delta},
$$
and hence we arrive at
\be \notag
\big\| \widehat{Z}^I \big( (\psi^* H \psi) \psi  \big) \big\|_{L^2_f(\Hcal_\tau)} 
\lesssim
 (C_1 \eps)^3 \tau^{-1+\delta},
 \qquad
 |I|\leq N-1.
\ee

\noindent\textbf{Step 2:} Next we estimate $\big\| \widehat{Z}^I \big( \gamma^\nu v \del_\nu (v \psi) \big) \big\|_{L^2_f(\Hcal_\tau)}$ for $|I|\leq N-1$.
We note that
$$
\aligned
\big\| \widehat{Z}^I & \big( \gamma^\nu v \del_\nu (v \psi) \big) \big\|_{L^2_f(\Hcal_\tau)}
\\
&\lesssim
\sum_{\substack{|I_1|+|I_2|\leq N-3\\ |I_3|\leq |I|}} \big\| |Z^{I_1} v| |Z^{I_2} v| |Z^{I_3} \del \psi| \big\|_{L^2_f(\Hcal_\tau)}
+
\sum_{\substack{|I_1|+|I_3|\leq N-3\\ |I_2|\leq |I|}} \big\| |Z^{I_1} v| |Z^{I_2}  \del v| |Z^{I_3}  \psi| \big\|_{L^2_f(\Hcal_\tau)}
\\
& \quad+
\sum_{\substack{|I_1|\leq |I|\\ |I_2|+|I_3|\leq N-3}} \big\| |Z^{I_1} v| |Z^{I_2}  v| |Z^{I_3}  \psi| \big\|_{L^2_f(\Hcal_\tau)}
+
\sum_{\substack{|I_1|+|I_2|\leq N-3\\ |I_3|\leq |I|}} \big\| |Z^{I_1} v| |Z^{I_2}  v| |Z^{I_3} \psi| \big\|_{L^2_f(\Hcal_\tau)}
\\
& =:  \mathcal{B}_{2a} + \mathcal{B}_{2b} + \mathcal{B}_{2c} + \mathcal{B}_{2d}.
\endaligned
$$
For the term $\mathcal{B}_{2a}$, we have
$$
\aligned
\mathcal{B}_{2a}
\lesssim
\sum_{\substack{|I_1|+|I_2|\leq N-3\\ |I_3|\leq |I|}} \big\| (t/\tau) |Z^{I_1} v| |Z^{I_2} v| \big\|_{L^\infty(\Hcal_\tau)} \big\|(\tau/t) \del Z^{I_3}  \psi  \big\|_{L^2_f(\Hcal_\tau)}
\lesssim
(C_1 \eps)^3 \tau^{-2+2\delta},
\endaligned
$$
in which we used Lemma \ref{lem:commu}, Proposition \ref{prop:Linfty}, and Proposition \ref{prop:psi-t}.
For the term $\mathcal{B}_{2b}$, we get
$$
\aligned
\mathcal{B}_{2b}
\lesssim
\sum_{\substack{|I_1|+|I_3|\leq N-3\\ |I_2|\leq |I|}} \big\| (t/\tau) |Z^{I_1} v|  |Z^{I_3}  \psi| \big\|_{L^\infty(\Hcal_\tau)}  \big\| (\tau/t) \del Z^{I_2} v \big\|_{L^2_f(\Hcal_\tau)}
\lesssim
(C_1 \eps)^3 \tau^{-2+2\delta},
\endaligned
$$
in which we used Lemma \ref{lem:commu}, Proposition \ref{prop:L2}, and Proposition \ref{prop:Linfty}.
For the third term $\mathcal{B}_{2c}$, we obtain
$$
\aligned
\mathcal{B}_{2c}
\lesssim
\sum_{\substack{|I_1|\leq |I|\\ |I_2|+|I_3|\leq N-3}} \big\| Z^{I_1} v \big\|_{L^2_f(\Hcal_\tau)}  \big\| |Z^{I_2}  v| |Z^{I_3} \psi| \big\|_{L^\infty(\Hcal_\tau)}
\lesssim
(C_1 \eps)^3 \tau^{-2+2\delta},
\endaligned
$$
in which we used Lemma \ref{lem:commu}, Proposition \ref{prop:L2}, and Proposition \ref{prop:Linfty}.
For the last term $\mathcal{B}_{2d}$, we have
$$
\mathcal{B}_{2d}
\lesssim
\sum_{\substack{|I_1|+|I_2|\leq N-3\\ |I_3|\leq |I|}} \big\| (t/\tau) |Z^{I_1} v| |Z^{I_2} v| \big\|_{L^\infty(\Hcal_\tau)} \big\| (\tau/t) Z^{I_3} \psi  \big\|_{L^2_f(\Hcal_\tau)}
\lesssim
(C_1 \eps)^3 \tau^{-2+2\delta},
$$
in which we used Lemma \ref{lem:commu}, Proposition \ref{prop:L2}, and Proposition \ref{prop:Linfty}.

To conclude, we have
\be \notag
\big\| \widehat{Z}^I \big( \gamma^\nu v \del_\nu (v \psi) \big) \big\|_{L^2_f(\Hcal_\tau)}
\lesssim
(C_1 \eps)^3 \tau^{-2+2\delta},
\qquad
|I|\leq N-1.
\ee

\noindent\textbf{Step 3:} We now turn to $\big\| \widehat{Z}^I \big( \del_\alpha v \del^\alpha \psi \big) \big\|_{L^2_f(\Hcal_\tau)}$ for $|I|\leq N-2$.
Recalling Lemmas \ref{lem:commu} and \ref{lem:null}, we find that
$$
\aligned
\big\| \widehat{Z}^I \big( \del_\alpha v \del^\alpha \psi \big) \big\|_{L^2_f(\Hcal_\tau)}
\lesssim
\sum_{|I_1|+|I_2|\leq |I|, a, b} \Big( &\big\| (\tau/t)^2 \del_t Z^{I_1} v \del_t Z^{I_2} \psi \big\|_{L^2_f(\Hcal_\tau)}
+
\big\| \del_t Z^{I_1} v \underline{\del}_a Z^{I_2} \psi \big\|_{L^2_f(\Hcal_\tau)}
\\
+
&\big\| \underline{\del}_a Z^{I_1} v \del_t Z^{I_2} \psi \big\|_{L^2_f(\Hcal_\tau)}
+
\big\| \underline{\del}_a Z^{I_1} v \underline{\del}_b Z^{I_2} \psi \big\|_{L^2_f(\Hcal_\tau)}
\Big)
\\
=: \mathcal{B}_{3a} + \mathcal{B}_{3b} + & \mathcal{B}_{3c} + \mathcal{B}_{3d}.
\endaligned
$$
We have
$$
\aligned
\mathcal{B}_{3a}
&\lesssim
\sum_{|I_1|\leq N-3\,, |I_2|\leq |I|} \big\| (\tau/t) (t-r)^{-1} \del_t Z^{I_1} v \big\|_{L^\infty(\Hcal_\tau)} \big\| (\tau/t) (t-r) \del_t Z^{I_2} \psi \big\|_{L^2_f(\Hcal_\tau)}
\\
& \quad + 
\sum_{|I_1|\leq |I|, |I_2|\leq N-4} \big\| (\tau/t) \del_t Z^{I_1} v \big\|_{L^2_f(\Hcal_\tau)} \big\| (\tau/t) \del_t Z^{I_2} \psi \big\|_{L^\infty(\Hcal_\tau)}
\\
&\lesssim
(C_1 \eps)^2 \tau^{-2+2\delta},
\endaligned
$$
in which we used Propositions \ref{prop:L2}, \ref{prop:Linfty}, and \ref{prop:psi-t}.
In succession, we get
$$
\aligned
\mathcal{B}_{3b} 
&\lesssim
\sum_{|I_1|\leq N-3, |I_2|\leq |I|, a} \big\| \tau^{-1} \del_t Z^{I_1} v \big\|_{L^\infty(\Hcal_\tau)}  \big\| (\tau/t) L_a Z^{I_2} \psi \big\|_{L^2_f(\Hcal_\tau)}
\\
& \quad + 
\sum_{|I_1|\leq |I|, |I_2|\leq N-4, a} \big\| (\tau/t) \del_t Z^{I_1} v \big\|_{L^2_f(\Hcal_\tau)} \big\| \tau^{-1} L_a Z^{I_2} \psi \big\|_{L^\infty(\Hcal_\tau)}
\\
&\lesssim
(C_1 \eps)^2 \tau^{-2+2\delta},
\endaligned
$$
in which we used Lemma \ref{lem:commu}, Proposition \ref{prop:L2}, and Proposition \ref{prop:Linfty}.
Similarly, we obtain
$$
\aligned
\mathcal{B}_{3c}
&\lesssim
\sum_{|I_1|\leq N-4, |I_2|\leq |I|, a} \big\| (t/\tau) (t-r)^{-1} t^{-1} L_a Z^{I_1} v \big\|_{L^\infty(\Hcal_\tau)} \big\| (\tau/t) (t-r) \del_t Z^{I_2} \psi \big\|_{L^2_f(\Hcal_\tau)}
\\
& \quad+
\sum_{|I_1|\leq |I|, |I_2|\leq N-4, a} \big\| L_a Z^{I_1} v \big\|_{L^2_f(\Hcal_\tau)} \big\| t^{-1} \del_t Z^{I_2} \psi \big\|_{L^\infty(\Hcal_\tau)}
\\
&\lesssim
(C_1 \eps)^2 \tau^{-2+2\delta},
\endaligned
$$
in which we used Lemma \ref{lem:commu}, Proposition \ref{prop:L2}, Proposition \ref{prop:Linfty}, and Proposition \ref{prop:psi-t}.
Easily, we get
$$
\mathcal{B}_{3d}
\lesssim
(C_1 \eps)^2 \tau^{-2}.
$$
To conclude, we get
\be \notag
\big\| \widehat{Z}^I \big( \del_\alpha v \del^\alpha \psi \big) \big\|_{L^2_f(\Hcal_\tau)}
\lesssim
(C_1 \eps)^2 \tau^{-2+2\delta},
\qquad
|I|\leq N-2.
\ee

\noindent\textbf{Step 4:} Finally, we estimate $\big\| \widehat{Z}^I \big( \del_\alpha v \del^\alpha \psi \big) \big\|_{L^2_f(\Hcal_\tau)}$ for $|I|\leq N-1$.
The estimate is very similar to \textbf{Step 3} above, but we write it out for completeness. We first bound
$$
\aligned
\big\| \widehat{Z}^I \big( \del_\alpha v \del^\alpha \psi \big) \big\|_{L^2_f(\Hcal_\tau)}
\lesssim
\sum_{|I_1|+|I_2|\leq |I|, a, b} \Big( &\big\| (\tau/t)^2 \del_t Z^{I_1} v \del_t Z^{I_2} \psi \big\|_{L^2_f(\Hcal_\tau)}
+
\big\| \del_t Z^{I_1} v \underline{\del}_a Z^{I_2} \psi \big\|_{L^2_f(\Hcal_\tau)}
\\
+
&\big\| \underline{\del}_a Z^{I_1} v \del_t Z^{I_2} \psi \big\|_{L^2_f(\Hcal_\tau)}
+
\big\| \underline{\del}_a Z^{I_1} v \underline{\del}_b Z^{I_2} \psi \big\|_{L^2_f(\Hcal_\tau)}
\Big)
\\
=: \mathcal{B}_{4a} + \mathcal{B}_{4b} + & \mathcal{B}_{4c} + \mathcal{B}_{4d}.
\endaligned
$$
We have
$$
\aligned
\mathcal{B}_{4a}
&\lesssim
\sum_{|I_1|\leq N-4, |I_2|\leq |I|} \big\| (\tau/t)  \del_t Z^{I_1} v \big\|_{L^\infty(\Hcal_\tau)} \big\| (\tau/t)  \del_t Z^{I_2} \psi \big\|_{L^2_f(\Hcal_\tau)}
\\
& \quad +
\sum_{|I_1|\leq |I|, |I_2|\leq N-4} \big\| (\tau/t) \del_t Z^{I_1} v \big\|_{L^2_f(\Hcal_\tau)} \big\| (\tau/t) \del_t Z^{I_2} \psi \big\|_{L^\infty(\Hcal_\tau)}
\\
&\lesssim
(C_1 \eps)^2 \tau^{-1+\delta},
\endaligned
$$
in which we used Propositions \ref{prop:L2}, \ref{prop:Linfty}, and \ref{prop:psi-t}.
In succession, we get
$$
\aligned
\mathcal{B}_{4b} 
&\lesssim
\sum_{|I_1|\leq N-4, |I_2|\leq |I|, a} \big\| \tau^{-1} (t-r) \del_t Z^{I_1} v \big\|_{L^\infty(\Hcal_\tau)}  \big\| (\tau/t) (t-r)^{-1} L_a Z^{I_2} \psi \big\|_{L^2_f(\Hcal_\tau)}
\\
& \quad + 
\sum_{|I_1|\leq |I|, |I_2|\leq N-4, a} \big\| (\tau/t) \del_t Z^{I_1} v \big\|_{L^2_f(\Hcal_\tau)} \big\| \tau^{-1} L_a Z^{I_2} \psi \big\|_{L^\infty(\Hcal_\tau)}
\\
&\lesssim
(C_1 \eps)^2 \tau^{-1+\delta},
\endaligned
$$
in which we used Lemma \ref{lem:commu}, Proposition \ref{prop:L2}, and Proposition \ref{prop:Linfty}.
Similarly, we obtain
$$
\aligned
\mathcal{B}_{4c}
&\lesssim
\sum_{|I_1|\leq N-4, |I_2|\leq |I|, a} \big\| (t/\tau) t^{-1} L_a Z^{I_1} v \big\|_{L^\infty(\Hcal_\tau)} \big\| (\tau/t) \del_t Z^{I_2} \psi \big\|_{L^2_f(\Hcal_\tau)}
\\
& \quad + \sum_{|I_1|\leq |I|, |I_2|\leq N-4, a} \big\| L_a Z^{I_1} v \big\|_{L^2_f(\Hcal_\tau)} \big\| t^{-1} \del_t Z^{I_2} \psi \big\|_{L^\infty(\Hcal_\tau)}
\\
&\lesssim
(C_1 \eps)^2 \tau^{-2+2\delta},
\endaligned
$$
in which we used Lemma \ref{lem:commu}, Proposition \ref{prop:L2}, Proposition \ref{prop:Linfty}, and Proposition \ref{prop:psi-t}.
Easily, we get
$$
\mathcal{B}_{4d}
\lesssim
(C_1 \eps)^2 \tau^{-2}.
$$
To conclude, we have
\be \notag
\big\| \widehat{Z}^I \big( \del_\alpha v \del^\alpha \psi \big) \big\|_{L^2_f(\Hcal_\tau)}
\lesssim
(C_1 \eps)^2 \tau^{-1+\delta},
\qquad
|I|\leq N-1.
\ee

\end{proof}

\subsection{Improved estimates for low order energy}
In order to improve the lower order energy bounds for Klein-Gordon and Dirac fields, we use nonlinear transformations (see Sections \ref{subsec:hidden-null} and \ref{subsec:nonlinearTrans}) to remove the slowly-decaying terms. This is at the expense of introducing null and cubic terms yet nevertheless allows us to obtain the desired energy bounds. Our strategy is to first estimate the new variables $\widetilde{v}, \widetilde{\psi}$ in Lemmas \ref{lem:transf-v} and \ref{lemma:EoM-tilde-psi}, and then use these to estimate the original unknowns $v, \psi$.

\begin{lemma}\label{lem:low-KGt}
We have
\be\notag
\aligned
\Ecal_1 (s, Z^I \widetilde{v})^{1/2}
&\lesssim 
\left\lbrace
\begin{array}{ll}
\eps + (C_1 \eps)^{3/2},&|I| \leq N-2,
\vspace{0.15cm}
\\
\eps + (C_1 \eps)^{3/2} s^\delta,&|I| \leq N-1.
\end{array}
\right.
\endaligned
\ee

%\be 
%\aligned
%\Ecal_1 (s, Z^I \widetilde{v})^{1/2}
%&\lesssim \eps + (C_1 \eps)^{3/2},
%\qquad
%&|I| \leq N-2,
%\\
%\Ecal_1 (s, Z^I \widetilde{v})^{1/2}
%&\lesssim \eps + (C_1 \eps)^{3/2} s^\delta,
%\qquad
%&|I| \leq N-1.
%\endaligned
%\ee
\end{lemma}
\begin{proof}
Using the energy estimate in Proposition \ref{prop:energy-ineq-KG} for Klein-Gordon equations, together with the estimates in Lemma \ref{lem:NL-KGt}, we get for the $\widetilde{v}$ component that
\begin{eqnarray}\notag
\aligned
\Ecal_1 &(s, Z^I \widetilde{v})^{1/2}
\\
&\lesssim
\Ecal_1 (s_0, Z^I \widetilde{v})^{1/2}
+
\int_{s_0}^s \big\| Z^I \big(- i \del_\nu(v\psi^*) (H \gamma^\nu)^* \psi + i \psi^* H \gamma^\nu \del_\nu (v\psi) + 2 \del_\alpha \psi^* H \del^\alpha \psi \big) \big\|_{L^2_f(\Hcal_\tau)} \, \di\tau
\\
&\lesssim
\left\{
\begin{array}{lll}
\eps + (C_1 \eps)^2, \qquad &|I| \leq N-2,
\vspace{0.15cm}
\\
\eps + (C_1 \eps)^2 s^\delta, \qquad &|I| \leq N-1.
\end{array}
\right.
\endaligned
\end{eqnarray}
\end{proof}

\begin{proposition}\label{prop:low-KG}
We have
\be \notag
\aligned
\Ecal_1 (s, Z^I {v})^{1/2}
&\lesssim 
\left\lbrace
\begin{array}{ll}
\eps + (C_1 \eps)^{3/2}, &|I| \leq N-2,
\vspace{0.15cm}
\\
\eps + (C_1 \eps)^{3/2} s^\delta, & |I| \leq N-1.
\end{array}
\right.
\endaligned
\ee

%\be 
%\aligned
%\Ecal_1 (s, Z^I {v})^{1/2}
%&\lesssim \eps + (C_1 \eps)^{3/2},
%\qquad
%&|I| \leq N-2,
%\\
%\Ecal_1 (s, Z^I {v})^{1/2}
%&\lesssim \eps + (C_1 \eps)^{3/2} s^\delta,
%\qquad
%&|I| \leq N-1.
%\endaligned
%\ee
\end{proposition}
\begin{proof}
We note that
$$
\Ecal_1 (s, Z^I {v})^{1/2}
\lesssim
\Ecal_1 (s, Z^I \widetilde{v})^{1/2}
+
\Ecal_1 \big(s, Z^I (\psi^* H \psi)\big)^{1/2},
$$
so we only need to bound $\Ecal_1 \big(s, Z^I (\psi^* H \psi) \big)^{1/2}$.
For $|I| \leq N-2$, we know that
$$
\aligned
\Ecal_1 &\big(s, Z^I (\psi^* H \psi) \big)^{1/2}
\\
&\lesssim
\big\| (s/t) \del_t Z^I (\psi^* H \psi) \big\|_{L^2_f(\Hcal_s)}
+
\sum_a \big\| \underline{\del}_a Z^I (\psi^* H \psi) \big\|_{L^2_f(\Hcal_s)}
+
\big\| Z^I (\psi^* H \psi) \big\|_{L^2_f(\Hcal_s)}
\\&=:
 \mathcal{B}_{1a} + \mathcal{B}_{1b} + \mathcal{B}_{1c}.
\endaligned
$$
We find that
$$
\aligned
\mathcal{B}_{1a}
\lesssim
\sum_{|I_1| \leq N-1, |I_2|\leq N-3} \big\| (s/t) Z^{I_1} \psi \big\|_{L^2_f(\Hcal_s)}  \big\|Z^{I_2} \psi \big\|_{L^\infty(\Hcal_s)}
\lesssim
(C_1 \eps)^2 s^{-1+2\delta},
\endaligned
$$
in which we used Proposition \ref{prop:L2} and Proposition \ref{prop:Linfty}.
To proceed, we have
$$
\aligned
\mathcal{B}_{1b}
\lesssim
\sum_{|I_1| \leq N-1, |I_2|\leq N-3} \big\| (s/t) Z^{I_1} \psi \big\|_{L^2_f(\Hcal_s)}  \big\|s^{-1} Z^{I_2} \psi \big\|_{L^\infty(\Hcal_s)}
\lesssim
(C_1 \eps)^2 s^{-2+2\delta},
\endaligned
$$
in which we used Lemma \ref{lem:commu}, Proposition \ref{prop:L2}, and Proposition \ref{prop:Linfty}.
We also get
$$
\aligned
\mathcal{B}_{1c}
\lesssim
\sum_{|I_1| \leq N-2, |I_2|\leq N-4} \big\| (s/t) Z^{I_1} \psi \big\|_{L^2_f(\Hcal_s)}  \big\|(t/s) Z^{I_2} \psi \big\|_{L^\infty(\Hcal_s)}
\lesssim
(C_1 \eps)^2,
\endaligned
$$
in which we used Proposition \ref{prop:L2} and Proposition \ref{prop:Linfty}.
Thus, we get
$$
\Ecal_1 (s, Z^I {v})^{1/2}
\lesssim
\Ecal_1 (s, Z^I \widetilde{v})^{1/2}
+
\Ecal_1 \big(s, Z^I (\psi^* H \psi)\big)^{1/2}
\lesssim
\eps + (C_1 \eps)^2,
\qquad
|I| \leq N-2.
$$

In a similar way, we get 
$$
\Ecal_1 (s, Z^I {v})^{1/2}
\lesssim
\eps + (C_1 \eps)^2 s^\delta,
\qquad
|I| \leq N-1.
$$
\end{proof}

\begin{lemma}\label{lem:low-Dt}
We have
\be \notag
\aligned
\Ecal^D (s, \widehat{Z}^I \widetilde{\psi})^{1/2}
&\lesssim 
\left\lbrace
\begin{array}{ll}
\eps + (C_1 \eps)^{3/2}, & |I| \leq N-2,
\vspace{0.15cm}
\\
\eps + (C_1 \eps)^{3/2} s^\delta, & |I| \leq N-1.
\end{array}
\right.
\endaligned
\ee

%\be 
%\aligned
%\Ecal^D (s, \widehat{Z}^I \widetilde{\psi})^{1/2}
%&\lesssim 
%\eps + (C_1 \eps)^{3/2},
%\qquad
%&|I| \leq N-2,
%\\
%\Ecal^D (s, \widehat{Z}^I \widetilde{\psi})^{1/2}
%&\lesssim \eps + (C_1 \eps)^{3/2} s^\delta,
%\qquad
%&|I| \leq N-1.
%\endaligned
%\ee
\end{lemma}
\begin{proof}
According to the energy estimate \eqref{eq:D-E1} for Dirac equations, we have
\begin{eqnarray}\notag
\aligned
\Ecal^D &(s, \widehat{Z}^I \widetilde{\psi})^{1/2}
\\
&\lesssim
\Ecal^D (s_0, \widehat{Z}^I \widetilde{\psi})^{1/2}
+
\int_{s_0}^s \big\|  \widehat{Z}^I \big( \big(\psi^* H \psi\big) \psi + i \gamma^\nu v \del_\nu (v \psi) - 2 \del_\alpha v \del^\alpha \psi \big) \big\|_{L^2_f(\Hcal_\tau)} \, \di\tau
\\
&\lesssim
\eps + (C_1 \eps)^2 s^\delta, \qquad |I| \leq N-1,
%&\left\{
%\begin{array}{lll}
%\eps^2 + (C_1 \eps)^3, \qquad &|I| \leq N-2,
%\vspace{0.15cm}
%\\
%\eps^2 + (C_1 \eps)^3 s^\delta, \qquad &|I| \leq N-1,
%\end{array}
%\right.
\endaligned
\end{eqnarray}
in which we used the estimates in Lemma \ref{lem:NL-Dt}.
As a consequence, we obtain
\bel{eq:psi-t-}
\aligned
\| (s/t) \widehat{Z}^I \widetilde{\psi}\|_{L^2_f(\Hcal_s)} 
+\| (\widehat{Z}^I \widetilde{\psi})_-\|_{L^2_f(\Hcal_s)} 
\lesssim \eps + (C_1 \eps)^2 s^\delta, \qquad |I| \leq N-1,
\\
|(\widehat{Z}^I \widetilde{\psi})_-|
\lesssim \big( \eps + (C_1 \eps)^2\big) t^{-1} s^{\delta}, \qquad |I| \leq N-3.
\endaligned
\ee
On the other hand, for $|I|\leq N-2$, we apply the energy estimate \eqref{eq:D-E2} for Dirac equations to get
$$
\aligned
\Ecal^D (s, \widehat{Z}^I \widetilde{\psi})
&\lesssim
\Ecal^D (s_0, \widehat{Z}^I \widetilde{\psi})
+
\int_{s_0}^s \big\| (\tau/t) \big( \widehat{Z}^I \widetilde{\psi} \big)^* \gamma^0 \widehat{Z}^I \big( \big(\psi^* H \psi\big) \psi + i \gamma^\nu v \del_\nu (v \psi) - 2 \del_\alpha v \del^\alpha \psi \big) \big\|_{L^1_f(\Hcal_\tau)} \, \di\tau
\\
&\lesssim
\Ecal^D (s_0, \widehat{Z}^I \widetilde{\psi})
+
\int_{s_0}^s \big\|(\tau/t) \big( \widehat{Z}^I \widetilde{\psi} \big)^* \gamma^0 \widehat{Z}^I \big( \big(\psi^* H \psi\big) \psi \big) \big\|_{L^1_f(\Hcal_\tau)} \, \di\tau
\\
&\quad +
\int_{s_0}^s \big\| (\tau/t)\widehat{Z}^I \widetilde{\psi} \big\|_{L^2_f(\Hcal_\tau)} \big\| \widehat{Z}^I \big(  i \gamma^\nu v \del_\nu (v \psi) - 2 \del_\alpha v \del^\alpha \psi \big) \big\|_{L^2_f(\Hcal_\tau)} \, \di\tau
\\
&=:
 \mathcal{D}_1 + \mathcal{D}_2 + \mathcal{D}_3.
\endaligned
$$
For the term $\mathcal{D}_3$, the estimates in Lemma \ref{lem:NL-Dt} and \eqref{eq:psi-t-} imply that
$$
\mathcal{D}_3
\lesssim
(C_1 \eps)^3. 
$$
Then we treat the term $\mathcal{D}_2$, and according to Lemma \ref{lem:hidden-KG} we find that
$$
\aligned
\mathcal{D}_2
&\lesssim
\sum_{|I_1|+|I_2|\leq|I|} \int_{s_0}^s \big\|(\tau/t) \big( \widehat{Z}^I \widetilde{\psi} \big)^* \gamma^0  \widehat{Z}^{I_1} \psi Z^{I_2} \big(\psi^* H \psi\big)  \big\|_{L^1_f(\Hcal_\tau)} \, \di\tau
\\
&\lesssim
\sum_{|I_1|+|I_2|\leq|I|} \int_{s_0}^s \Big\|(\tau/t) \big|\big( \widehat{Z}^I \widetilde{\psi} \big)_-\big|  \big|\big(\widehat{Z}^{I_1} \psi\big)_+\big| \big|Z^{I_2} \big(\psi^* H \psi\big)\big|  \Big\|_{L^1_f(\Hcal_\tau)} \, \di\tau
\\
&\quad +
\sum_{|I_1|+|I_2|\leq|I|} \int_{s_0}^s \Big\|(\tau/t) \big|\big( \widehat{Z}^I \widetilde{\psi} \big)_+ \big|  \big| \big(\widehat{Z}^{I_1} \psi\big)_-\big| \big|Z^{I_2} \big(\psi^* H \psi\big)\big|  \Big\|_{L^1_f(\Hcal_\tau)} \, \di\tau
\\
&\quad +\sum_{|I_1|+|I_2|\leq|I|}  \int_{s_0}^s \big\|(\tau/t) \big|\big( \widehat{Z}^I \widetilde{\psi} \big)_-\big|  \big|\big( \widehat{Z}^{I_1} \psi\big)_-\big| \big|Z^{I_2} \big(\psi^* H \psi\big)\big|  \big\|_{L^1_f(\Hcal_\tau)} \, \di\tau
\\
& \quad + \sum_{|I_1|+|I_2|\leq|I|} \int_{s_0}^s \big\|(\tau/t)^3 \big| \widehat{Z}^I \widetilde{\psi} \big|  \big| \widehat{Z}^{I_1} \psi\big| \big|Z^{I_2} \big(\psi^* H \psi\big)\big|  \big\|_{L^1_f(\Hcal_\tau)} \, \di\tau
\\
&=: \mathcal{D}_{2a} + \mathcal{D}_{2b} + \mathcal{D}_{2c} + \mathcal{D}_{2d}.
\endaligned
$$
We proceed to have (recall $\big|\big(\widehat{Z}^{I_1} \psi\big)_+\big|\lesssim \big|\widehat{Z}^{I_1} \psi\big|$)
$$
\aligned
\mathcal{D}_{2a}
&\lesssim
\sum_{\substack{|I_1|\leq |I|,\\ |I_2|\leq N-3}} \int_{s_0}^s  \big\|\big( \widehat{Z}^I \widetilde{\psi} \big)_-\big\|_{L^2_f(\Hcal_\tau)}  \big\|(\tau/t) \big(\widehat{Z}^{I_1} \psi\big)_+\big\|_{L^2_f(\Hcal_\tau)}  \big\|Z^{I_2} \big(\psi^* H \psi\big)\big\|_{L^\infty(\Hcal_\tau)} \, \di\tau
\\
&\quad+
\sum_{\substack{|I_1|+|J_1|\leq N-3\\ |J_2|\leq |I|}} \int_{s_0}^s  \big\|\big( \widehat{Z}^I \widetilde{\psi} \big)_-\big\|_{L^2_f(\Hcal_\tau)}  \big\| \big(\widehat{Z}^{I_1} \psi\big)_+\big\|_{L^\infty(\Hcal_\tau)}  \big\|Z^{J_1} \psi \big\|_{L^\infty(\Hcal_\tau)}  \big\|(\tau/t) Z^{J_2} \psi \big\|_{L^2_f(\Hcal_\tau)} \, \di\tau
\\
&\lesssim
(C_1 \eps)^4 \int_{s_0}^s \tau^{-2+2\delta} \, d\tau
\lesssim
(C_1 \eps)^4,
\endaligned
$$
in which we used the estimates in \eqref{eq:psi-t-}, Proposition \ref{prop:L2}, and Proposition \ref{prop:Linfty}.
In turn we get, recall again $\big|\big(\widehat{Z}^{I} \psi\big)_+\big|\lesssim \big|\widehat{Z}^{I} \psi\big|$, that
$$
\aligned
\mathcal{D}_{2b}
&\lesssim
\sum_{|I_1|\leq |I|, |I_2|\leq N-3} \int_{s_0}^s  \big\|(\tau/t)\big( \widehat{Z}^I \widetilde{\psi} \big)_+\big\|_{L^2_f(\Hcal_\tau)}  \big\| \big(\widehat{Z}^{I_1} \psi\big)_-\big\|_{L^2_f(\Hcal_\tau)}  \big\|Z^{I_2} \big(\psi^* H \psi\big)\big\|_{L^\infty(\Hcal_\tau)} \, \di\tau
\\
&\quad+
\sum_{\substack{|I_1|+|J_1|\leq N-3\\ |J_2|\leq |I|}} \int_{s_0}^s  \big\|(\tau/t)\big( \widehat{Z}^I \widetilde{\psi} \big)_+\big\|_{L^2_f(\Hcal_\tau)}  \big\| (t/\tau)\big(\widehat{Z}^{I_1} \psi\big)_-\big\|_{L^\infty(\Hcal_\tau)}  \big\|Z^{J_1} \psi \big\|_{L^\infty(\Hcal_\tau)}  \big\|(\tau/t) Z^{J_1} \psi \big\|_{L^2_f(\Hcal_\tau)} \, \di\tau
\\
&\lesssim
(C_1 \eps)^4 \int_{s_0}^s \tau^{-2+3\delta} \, d\tau
\lesssim
(C_1 \eps)^4,
\endaligned
$$
in which we used again the estimates in \eqref{eq:psi-t-}, Proposition \ref{prop:L2}, and Proposition \ref{prop:Linfty}.
Since the analysis for bounding the other two terms is very similar, we write directly the final estimates without further details
$$
\mathcal{D}_{2c} + \mathcal{D}_{2d}
\lesssim
(C_1 \eps)^4.
$$
To sum things up, we have shown
\be \notag
\Ecal^D (s, \widehat{Z}^I \widetilde{\psi})
\lesssim
\eps^2 + (C_1 \eps)^3,
\qquad
|I|\leq N-2,
\ee
and thus the proof is complete.
\end{proof}

\begin{proposition}\label{prop:low-D}
We have
\be \notag
\aligned
\Ecal^D (s, \widehat{Z}^I {\psi})^{1/2}
&\lesssim 
\left\lbrace
\begin{array}{ll}
\eps + (C_1 \eps)^{3/2},&|I| \leq N-2,
\vspace{0.15cm}
\\
\eps + (C_1 \eps)^{3/2} s^\delta,&|I| \leq N-1.
\end{array}
\right.
\endaligned
\ee

%\be 
%\aligned
%\Ecal^D (s, \widehat{Z}^I {\psi})^{1/2}
%&\lesssim \eps + (C_1 \eps)^{3/2},
%\qquad
%&|I| \leq N-2,
%\\
%\Ecal^D (s, \widehat{Z}^I {\psi})^{1/2}
%&\lesssim \eps + (C_1 \eps)^{3/2} s^\delta,
%\qquad
%&|I| \leq N-1.
%\endaligned
%\ee
\end{proposition}
\begin{proof}
We recall that %\marginpar{Not correct!}
$$
\Ecal^D (s, \widehat{Z}^I {\psi})^{1/2}
\lesssim
\Ecal^D (s, \widehat{Z}^I \widetilde{\psi})^{1/2}
+
\Ecal^D \big(s, \widehat{Z}^I \big(\gamma^\nu \del_\nu (v \psi) \big)  \big)^{1/2},
$$
so it suffices to show
\begin{eqnarray}
\aligned
\Ecal^D \big(s, \widehat{Z}^I \big(\gamma^\nu \del_\nu (v \psi) \big)  \big)^{1/2}
\lesssim
% (C_1 \eps)^2, \qquad |I| \leq N-1.
\left\{
\begin{array}{lll}
 (C_1 \eps)^2, \qquad &|I| \leq N-2,
\vspace{0.15cm}
\\
 (C_1 \eps)^2 s^\delta, \qquad &|I| \leq N-1.
\end{array}
\right.
\endaligned
\end{eqnarray}
By the definition and decomposition of the energy functional $\Ecal^D$ in \eqref{eq:D-fctnal-1a}--\eqref{eq:D-fctnal-2a}, we need to bound
$$
\aligned
\Ecal^D \big(s, \widehat{Z}^I \big( \del (v \psi) \big)  \big)^{1/2}
\lesssim
\big\| \widehat{Z}^I \big( \del (v \psi) \big) \big\|_{L^2_f(\Hcal_s)}.
\endaligned
$$

We only estimate for the case of $|I|\leq N-2$ as the case of $|I|=N-1$ can be bounded in a very similar way. For $|I|\leq N-2$ we have
$$
\aligned
\big\| \widehat{Z}^I \big( \del (v \psi) \big) \big\|_{L^2_f(\Hcal_s)}
\lesssim
\sum_{|I_1|+|I_2|\leq |I|} \Big( \big\| Z^{I_1} \del v Z^{I_2} \psi \big\|_{L^2_f(\Hcal_s)}
+ 
 \big\| Z^{I_1}  v Z^{I_2} \del \psi \big\|_{L^2_f(\Hcal_s)} \Big)
 =: \mathcal{C}_1 + \mathcal{C}_2.
\endaligned
$$
We proceed to get
$$
\aligned
\mathcal{C}_1
&\lesssim
\sum_{\substack{|I_1|\leq |I|\\ |I_2|\leq N-3}}  \big\| Z^{I_1} \del v Z^{I_2} \psi \big\|_{L^2_f(\Hcal_s)}
+
\sum_{\substack{|I_1|\leq N-4\\ |I_2|\leq |I|}} \big\| Z^{I_1} \del v Z^{I_2} \psi \big\|_{L^2_f(\Hcal_s)}
\\
&\lesssim
\sum_{\substack{|J|\leq |I|+1\\ |I_2|\leq N-3}}  \big\| Z^{J}  v \big\|_{L^2_f(\Hcal_s)} \big\| Z^{I_2} \psi \big\|_{L^\infty(\Hcal_s)}
+
\sum_{\substack{|J|\leq N-3\\ |I_2|\leq |I|}} \big\| (t/s) Z^{J} v \big\|_{L^\infty(\Hcal_s)} \big\| (s/t) Z^{I_2} \psi \big\|_{L^2_f(\Hcal_s)}
\\
&\lesssim
(C_1 \eps)^2 s^{-1+2\delta},
\endaligned
$$
in which we used the estimates in Lemma \ref{lem:commu}, Proposition \ref{prop:L2}, and Proposition \ref{prop:Linfty}. 
Next, we bound
$$
\aligned
\mathcal{C}_2
&\lesssim
\sum_{\substack{|I_1|\leq |I|\\ |I_2|\leq N-4}} \big\| Z^{I_1}  v Z^{I_2} \del \psi \big\|_{L^2_f(\Hcal_s)}
+
\sum_{\substack{|I_1|\leq N-3\\ |I_2|\leq |I|}} \big\| Z^{I_1}  v Z^{I_2} \del \psi \big\|_{L^2_f(\Hcal_s)}
\\
&\lesssim
\sum_{\substack{|I_1|\leq |I|\\ |I_2|\leq N-3}} \big\| Z^{I_1}  v\big\|_{L^2_f(\Hcal_s)} \big\| Z^{J} \psi \big\|_{L^2_f(\Hcal_s)}
+
\sum_{\substack{|I_1|\leq N-3\\ |J|\leq |I|+1}} \big\| (t/s) Z^{I_1}  v \big\|_{L^\infty(\Hcal_s)} \big\| (s/t) Z^{J} \psi \big\|_{L^2_f(\Hcal_s)}
\\
&\lesssim
(C_1 \eps)^2 s^{-1+2\delta},
\endaligned
$$
in which we used again the estimates in Lemma \ref{lem:commu}, Proposition \ref{prop:L2}, and Proposition \ref{prop:Linfty}. 
Thus we arrive at
$$
\aligned
\Ecal^D (s, \widehat{Z}^I {\psi})^{1/2}
\lesssim
(C_1 \eps)^2 s^{-1+2\delta},
\qquad
|I|\leq N-2.
\endaligned
$$

Analogously, we can show
$$
\aligned
\Ecal^D (s, \widehat{Z}^I {\psi})^{1/2}
\lesssim
(C_1 \eps)^2 s^{\delta},
\qquad
|I|\leq N-1.
\endaligned
$$
which concludes the proposition.
\end{proof}

\subsection{Improved estimates for the highest order energy}

Our goal now is to close the highest order energy bootstrap. An essential difference compared with the lower order energy estimates is that nonlinear transformations are invalid due to issues with regularity. It seems impossible to close the highest order bootstrap at the first glance of the nonlinearities. Fortunately, the special structure of the DKG system, the Klein-Gordon decomposition within the nonlinearities and our $(t-r)$ weighted energy estimate (see Proposition \ref{prop:GhostWeight}) will allow us to reach the desired goals.

\begin{proposition}\label{prop:top-KG}
We have
\be \notag
\aligned
\Ecal_1 (s, Z^I v)^{1/2}
&\lesssim \eps + (C_1 \eps)^{2} s^{1+\delta},
\quad
&|I|  &= N.
\endaligned
\ee
\end{proposition}
\begin{proof}
Recall the energy estimate for Klein-Gordon equations in Proposition \ref{prop:energy-ineq-KG}, and for $|I|= N$ we find that
$$
\aligned
\Ecal_1 (s, Z^I v)^{1/2}
\lesssim
\Ecal_1 (s_0, Z^I v)^{1/2}
+
\int_{s_0}^s \|Z^I (\psi^* H \psi) \|_{L^2_f(\Hcal_\tau)} \, \di\tau.
\endaligned
$$
Direct calculations show that
$$
\aligned
\|Z^I (\psi^* H \psi) \|_{L^2_f(\Hcal_\tau)}
&\lesssim
\sum_{I_1+I_2=I} \|(Z^{I_1} \psi)^* H Z^{I_2} \psi) \|_{L^2_f(\Hcal_\tau)}
\\
&\lesssim
\sum_{\substack{|I_1\leq N-4\\ |I_1| \leq |I|}} \|(t/\tau) (t-r) Z^{I_1} \psi\|_{L^\infty(\Hcal_\tau)} \|(\tau/t) (t-r)^{-1} Z^{I_2} \psi) \|_{L^2_f(\Hcal_\tau)}
\\
&\lesssim
(C_1 \eps)^2 \tau^{\delta},
\endaligned
$$
in which we used the estimates in Propositions \ref{prop:L2} and \ref{prop:Linfty}, and the fact
$$
\big\| (t/\tau) (t-r)  \tau^{-1} \big\|_{L^\infty(\Hcal_\tau \cap \Kcal)}
\lesssim 1.
$$
Thus we arrive at
$$
\Ecal_1 (s, Z^I v)^{1/2}
\lesssim
\epsilon
+
(C_1 \eps)^2 \int_{s_0}^s \tau^\delta \, \di\tau
\lesssim
\eps + (C_1 \eps)^2 s^{1+\delta}.
$$

\end{proof}

\begin{proposition}\label{prop:top-D}
We have
\be \notag
\aligned
\Ecal^D (s, \widehat{Z}^I  \psi, 1)^{1/2}
&\lesssim \eps + (C_1 \eps)^{3/2} s^\delta,
\quad
&|I| &= N.
\endaligned
\ee
\end{proposition}
\begin{proof}
We apply a $(t-r)$--weighted energy estimate for the Dirac equation of $\widehat{Z}^I \psi$ (with $|I|=N$) in Proposition \ref{prop:GhostWeight} with $\gamma=1$ to get
$$
\Ecal^D (s, \widehat{Z}^I  \psi, 1)
\lesssim
\Ecal^D (s_0, \widehat{Z}^I  \psi, 1)
+
\int_{s_0}^s \big\|(\tau/t) (t-r)^{-2} (\widehat{Z}^I \psi)^* \gamma^0 \widehat{Z}^I (v\psi) \big\|_{L^1_f(\Hcal_\tau)} \, \di\tau.
$$
We apply Lemma \ref{lem:hidden-KG} to get
$$
\aligned
\big\|(\tau/t)& (t-r)^{-2} (\widehat{Z}^I \psi)^* \gamma^0 \widehat{Z}^I (v\psi) \big\|_{L^1_f(\Hcal_\tau)}
\\
&\lesssim
\sum_{|I_1|+|I_2| \leq |I|} \big\| (\tau/t) (t-r)^{-2} (\widehat{Z}^I \psi)^* \gamma^0 Z^{I_1} v\widehat{Z}^{I_2} \psi \big\|_{L^1_f(\Hcal_\tau)}
\\
&\lesssim
\sum_{|I_1|+|I_2| \leq |I|} \big\| (\tau/t) (t-r)^{-2} |Z^{I_1} v| |(\widehat{Z}^I \psi)_-|  |(\widehat{Z}^{I_2} \psi)_-|  \big\|_{L^1_f(\Hcal_\tau)}
\\
&\quad +
\sum_{|I_1|+|I_2| \leq |I|} \big\| (\tau/t) (t-r)^{-2} |Z^{I_1} v| |(\widehat{Z}^I \psi)_-|  |(\widehat{Z}^{I_2} \psi)_+|  \big\|_{L^1_f(\Hcal_\tau)}
\\
&\quad + 
\sum_{|I_1|+|I_2| \leq |I|} \big\|(\tau/t) (t-r)^{-2} |Z^{I_1} v|  |(\widehat{Z}^I \psi)_+|  |(\widehat{Z}^{I_2} \psi)_-|  \big\|_{L^1_f(\Hcal_\tau)}
\\
&\quad + 
\sum_{|I_1|+|I_2| \leq |I|} \big\| (\tau/t) (t-r)^{-2} |Z^{I_1} v| (\tau/t)^2 |\widehat{Z}^I \psi|  |\widehat{Z}^{I_2} \psi| \big\|_{L^1_f(\Hcal_\tau)}
\\
&=:
\mathcal{A}_1 + \mathcal{A}_2 + \mathcal{A}_3 + \mathcal{A}_4.
\endaligned
$$
We next estimate each of these four terms.

We start with the term $\mathcal{A}_1$, and we first decompose it into two parts
$$
\aligned
\mathcal{A}_1
&\leq
\sum_{|I_1|+|I_2| \leq |I|, |I_1|\leq |I_2|} \big\| (\tau/t) (t-r)^{-2} |Z^{I_1} v| |(\widehat{Z}^I \psi)_-|  |(\widehat{Z}^{I_2} \psi)_-|  \big\|_{L^1_f(\Hcal_\tau)}
\\
&\quad + 
\sum_{|I_1|+|I_2| \leq |I|, |I_1|\geq |I_2|} \big\| (\tau/t) (t-r)^{-2} |Z^{I_1} v| |(\widehat{Z}^I \psi)_-|  |(\widehat{Z}^{I_2} \psi)_-|  \big\|_{L^1_f(\Hcal_\tau)}
\\
&=\mathcal{A}_{1a} + \mathcal{A}_{1b}.
\endaligned
$$
In conjunction, we further get
$$
\aligned
\mathcal{A}_{1a}
\lesssim
\sum_{\substack{|I_1| \leq N-4\\ |I_2|\leq |I|}} \big\|(\tau/t)  Z^{I_1} v\big\|_{L^\infty(\Hcal_\tau)} \Big\| {(\widehat{Z}^I \psi)_- \over (t-r)}\Big\|_{L^2_f(\Hcal_\tau)}  \Big\| {(\widehat{Z}^{I_2} \psi)_- \over (t-r)} \Big\|_{L^2_f(\Hcal_\tau)}
\lesssim
(C_1 \eps)^3 \tau^{-1+2\delta},
\endaligned
$$
in which we used Propositions \ref{prop:L2} and \ref{prop:Linfty}. We also find
$$
\aligned
\mathcal{A}_{1b}
\lesssim
\sum_{\substack{|I_1| \leq |I|\\ |I_2|\leq N-4}}  \big\|Z^{I_1} v\big\|_{L^2_f(\Hcal_\tau)}  \Big\| {(\widehat{Z}^I \psi)_- \over (t-r)}\Big\|_{L^2_f(\Hcal_\tau)}  \Big\|{(\tau/t) (\widehat{Z}^{I_2} \psi)_- \over (t-r)} \Big\|_{L^\infty(\Hcal_\tau)}
\lesssim
(C_1 \eps)^3 \tau^{-1+2\delta},
\endaligned
$$
in which we used Propositions \ref{prop:L2} and \ref{prop:Linfty}, as well as the fact that
$$
\big\| (\tau/t) (t-r)^{-1} t^{-1} \big\|_{L^\infty(\Hcal_\tau \cap \Kcal)}
\lesssim \tau^{-2}.
$$
Thus we get
$$
\mathcal{A}_1
\lesssim (C_1 \eps)^3 \tau^{-1+2\delta}.
$$

Next, we bound the term $\mathcal{A}_2$ as
$$
\aligned
\mathcal{A}_2
&\leq
\sum_{|I_1|+|I_2| \leq |I|, |I_1|\leq |I_2|} \big\| (\tau/t) (t-r)^{-2} |Z^{I_1} v| |(\widehat{Z}^I \psi)_-|  |(\widehat{Z}^{I_2} \psi)_+|  \big\|_{L^1_f(\Hcal_\tau)}
\\
&\quad + 
\sum_{|I_1|+|I_2| \leq |I|, |I_1|\geq |I_2|} \big\| (\tau/t) (t-r)^{-2} |Z^{I_1} v| |(\widehat{Z}^I \psi)_-|  |(\widehat{Z}^{I_2} \psi)_+|  \big\|_{L^1_f(\Hcal_\tau)}
\\
&=\mathcal{A}_{2a} + \mathcal{A}_{2b}.
\endaligned
$$
To proceed, we have
$$
\aligned
\mathcal{A}_{2a}
\lesssim
\sum_{\substack{|I_1| \leq N-4\\ |I_2|\leq |I|}} \big\|  Z^{I_1} v\big\|_{L^\infty(\Hcal_\tau)} \Big\|{ (\widehat{Z}^I \psi)_- \over (t-r)}\Big\|_{L^2_f(\Hcal_\tau)}  \Big\|{ (\tau/t) \widehat{Z}^{I_2} \psi \over (t-r)} \Big\|_{L^2_f(\Hcal_\tau)}
\lesssim
(C_1 \eps)^3 \tau^{-1+2\delta},
\endaligned
$$
in which we used Propositions \ref{prop:L2} and \ref{prop:Linfty}, and
$$
\aligned
\mathcal{A}_{2b}
\lesssim
\sum_{\substack{|I_1|\leq |I|\\ |I_2|\geq N-4}}  \big\|Z^{I_1} v\big\|_{L^2_f(\Hcal_\tau)}  \Big\|{ (\widehat{Z}^I \psi)_- \over (t-r)}\Big\|_{L^2_f(\Hcal_\tau)}  \Big\|{(\tau/t)  \widehat{Z}^{I_2} \psi \over (t-r)}\Big\|_{L^\infty(\Hcal_\tau)}
\lesssim
(C_1 \eps)^3 \tau^{-1+2\delta},
\endaligned
$$
in which we used Propositions \ref{prop:L2} and \ref{prop:Linfty}, as well as the fact
$$
\big\| (\tau/t) (t-r)^{-1} \tau^{-1} \big\|_{L^\infty(\Hcal_\tau \cap \Kcal)}
\lesssim \tau^{-2}.
$$
Thus we obtain
$$
\mathcal{A}_2
\lesssim (C_1 \eps)^3 \tau^{-1+2\delta}.
$$

In a very similar manner to estimating the term $\mathcal{A}_2$, we can show
$$
\mathcal{A}_3 + \mathcal{A}_4
\lesssim
(C_1 \eps)^3 \tau^{-1+2\delta}.
$$

By gathering these estimates, we arrive at
$$
\Ecal^D (s, \widehat{Z}^I  \psi, 1)
\lesssim
\eps^2 + (C_1 \eps)^3 \int_{s_0}^s \tau^{-1+2\delta} \, \di\tau
\lesssim
\eps^2 + (C_1 \eps)^3 s^{2\delta},
\qquad
|I|= N.
$$

The proof is complete.
\end{proof}

\subsection{Proof of Theorem \ref{thm:main1}}

\begin{proof}
\noindent \textbf{Global existence and time decay.}
The results of Propositions \ref{prop:low-KG}, \ref{prop:low-D}, \ref{prop:top-KG}, and \ref{prop:top-D} imply that for a fixed $0<\delta\ll1$ and $\mathbb{N}\ni N \geq 7$ there exists an $\eps_0>0$ sufficiently small that for all $0<\eps\leq \eps_0$ we have
\bel{eq:BA-improved}
\aligned
\Ecal^D (s, \widehat{Z}^I \psi)^{1/2}
+
\Ecal_1 (s, Z^I v)^{1/2}
&\leq {1\over 2} C_1 \eps,
\quad
&|I|  &\leq N-2,  
\\
\Ecal^D (s, \widehat{Z}^I \psi)^{1/2}
+
\Ecal_1 (s, Z^I v)^{1/2}
&\leq {1\over 2} C_1 \eps s^\delta,
\quad
&|I|  &= N-1,  
\\
\Ecal^D (s, \widehat{Z}^I \psi, 1)^{1/2}
+
s^{-1} \Ecal_1 (s, Z^I v)^{1/2}
&\leq {1\over 2} C_1 \eps s^\delta,
\quad
&|I|  &= N.
\endaligned
\ee

We can now conclude the bootstrap argument. By classical local existence results for nonlinear hyperbolic PDEs, the bounds \eqref{eq:BA-Dirac} hold whenever the solution exists. Clearly $s_1>s_0$ and, moreover, if $s_1<+\infty$ then one of the inequalities in \eqref{eq:BA-Dirac} must be an equality. However we see from \eqref{eq:BA-improved} that by choosing $C_1$ sufficiently large and $\eps_0$ sufficiently small, the bounds \eqref{eq:BA-Dirac} are in fact refined. This then implies that we must have $s_1=+\infty$. 
 Finally the decay estimates \eqref{eq:sharp-decay} follow from \eqref{eq:BA-improved} combined with the Sobolev estimates \eqref{eq:Sobolev} and \eqref{eq:Sobolev4}. \vspace{5pt}
 
\noindent\textbf{Scattering.} 
We next  show the scattering of the solution $(v, \psi)$. We will only illustrate the proof for the Klein-Gordon field $v$, as the proof for the Dirac field $\psi$ is analogous. Due to Lemma \ref{lem:scattering}, it suffices to show that
$$
\int_{t_0}^{+\infty} \| \psi^* H \psi \|_{L^2(\RR^2)} \, \di t
<+\infty.
$$
However, this does not seem possible. So we instead show the scattering for the variable $\widetilde{\psi}$ in Lemma \ref{lem:transf-v}.
In any case, we need to first derive the bounds of $\| Z^I \psi \|_{L^2(\RR^2)}$ (i.e., on constant $t$-slices) from the known ones $\| Z^I \psi \|_{L^2_f(\Hcal_s)}$ (i.e., on constant $s= \sqrt{t^2-r^2}$-slices). To do so, for any large $T>t_0+2$ the conservation of charge implies that 
$$
\|\psi(T)\|_{L^2(\RR^2)}
=
\|\psi_0\|_{L^2(\RR^2)}
\lesssim
\eps.
$$
In addition, for the $\widehat{Z} \psi$ equation we integrate the differential identity
\be \notag
\aligned
\del_t &\big( (\widehat{Z} \psi)^* (\widehat{Z} \psi) \big)
+ \del_a \big( (\widehat{Z} \psi)^* \gamma^0 \gamma^a (\widehat{Z} \psi)  \big)
\\
&=
i (\widehat{Z} \psi)^* \gamma^0 \big( (Zv) \psi + v (\widehat{Z} \psi) \big)
-
i \big( (Zv) \psi + v (\widehat{Z} \psi) \big)^* \gamma^0 \widehat{Z} \psi
\endaligned
\ee
over the spacetime region $
R_0 \define \{ (t, x):  t \leq T, t^2 - |x|^2 \geq s_0^2  \} \cap \{(t, x): t \geq |x| + 1 \}$
to get
$$
\aligned
\|(\widehat{Z} \psi) (T)\|^2_{L^2(\RR^2)}
&\lesssim
\Ecal^D (s_0, \widehat{Z} \psi)
+
\int_{R_0} \big| (\widehat{Z} \psi)^* \gamma^0 \big( (Zv) \psi + v (\widehat{Z} \psi) \big) \big| \, \di x \di t
\\
&\lesssim
\eps^2
+
\int_{\Kcal_{[s_0, T]}} \big| (\widehat{Z} \psi)^* \gamma^0 \big( (Zv) \psi + v (\widehat{Z} \psi) \big) \big| \, \di x \di t.
\endaligned
$$
To proceed, we have
$$
\aligned
\int_{\Kcal_{[s_0, T]}} &\big| (\widehat{Z} \psi)^* \gamma^0 \big( (Zv) \psi + v (\widehat{Z} \psi) \big)\big| \, \di x \di t
\lesssim
\int_{s_0}^T \big\| (\tau/t) (\widehat{Z} \psi)^* \gamma^0 \big( (Zv) \psi + v (\widehat{Z} \psi) \big) \big\|_{L^1_f(\Hcal_\tau)} \, \di \tau
\\
&\lesssim
\sum_{|I|+|J|\leq 1} \int_{s_0}^T \big\| (\tau/t) \widehat{Z} \psi \big\|_{L^2_f(\Hcal_\tau)} \big\| \widehat{Z}^I \psi \big\|_{L^\infty(\Hcal_\tau)} \big\| Z^J v \big\|_{L^2_f(\Hcal_\tau)} \, \di \tau
\\
&\lesssim
(C_1 \eps)^3 \int_{s_0}^T \tau^{-1} \, \di \tau
\lesssim (C_1 \eps)^3 \log T.
\endaligned
$$
Next, we use Lemma \ref{lem:null} to bound
$$
\aligned
\| \del_\alpha &\psi H \del^\alpha \psi \|_{L^2(\RR^2)}
\\
&\lesssim
\|(s^2/t^2) |\del_t \psi|^2\|_{L^2(\RR^2)}
+
\sum_a \| |\del_t \psi| t^{-1} |L_a \psi| \|_{L^2(\RR^2)}
+
\sum_{a, b} \| t^{-1} |L_a \psi| t^{-1} |L_b \psi| \|_{L^2(\RR^2)}
\\
&\lesssim
(C_1 \eps)^2 t^{-3/2} \log t,
\endaligned
$$
which is an integrable quantity. 
Thus we get
$$
\int_{t_0}^{+\infty} \| \del_\alpha \psi H \del^\alpha \psi \|_{L^2(\RR^2)} \, \di t
<+\infty.
$$
Similarly,  we can show
$$
\int_{t_0}^{+\infty} \Big( \big\|\psi^* H \gamma^\nu \del_\nu (v\psi)\big\|_{L^2(\RR^2)} + \big\| \del_\alpha \psi H \del^\alpha \psi \big\|_{L^2(\RR^2)} \Big) \, \di t
<+\infty.
$$ 
Thus, there exists a free Klein-Gordon component ${v}^+$, such that
\be 
\lim_{t\to +\infty} \Big( \sum_\alpha \| \del_\alpha (\widetilde{v} - {v}^+) \|_{L^2(\RR^2)} + \| \widetilde{v} - {v}^+\|_{L^2(\RR^2)} \Big)
=0.
\ee
We note that for all $t\geq t_0$ it holds
$$
\sum_\alpha \| \del_\alpha (\psi^* H \psi) \|_{L^2(\RR^2)}
+
\| \psi^* H \psi \|_{L^2(\RR^2)}
\lesssim
(C_1 \eps)^2 t^{-1/2} \log t
\to 0,
\qquad
\text{as   } t\to +\infty.
$$ 
Finally, we conclude that
\be 
\lim_{t\to +\infty} \Big( \sum_\alpha \| \del_\alpha ({v} - {v}^+) \|_{L^2(\RR^2)} + \| {v} - {v}^+\|_{L^2(\RR^2)} \Big)
=0.
\ee 

The proof is complete.
\end{proof}

%==============================================================================================
\appendix
\section{Proof of Theorem \ref{thm:main2}}\label{sec:proof2} 

We note that the proof for Theorem \ref{thm:main2} is similar to, and even easier than, the proof of Theorem \ref{thm:main1}. Given this, we omit some details for certain estimates in the proof. 
 
\subsection{Bootstrap assumptions and preliminary estimates}
Fix $N\in\mathbb{N}$ a large integer ($N \geq 4$ will work for our argument below). The local well-posedness theory guarantees that there exists  $C_0>0$ such that the following bounds hold for all $|I|\leq N$:
\bea\label{eq:m1BApre}
\Ecal_1(s_0, Z^I v)^{1/2} + \Ecal^D (s_0, \widehat{Z}^I \psi)^{1/2} \leq C_0 \eps. 
\eea
Next we assume that the following bounds hold for $s \in [s_0, s_1)$:  
\bel{eq:BA-Dirac2}
\aligned
\Ecal^D (s, \widehat{Z}^I \psi)^{1/2}
+
\Ecal_1 (s, Z^I v)^{1/2}
&\leq C_1 \eps,
\quad
&|I|  &\leq N-1,  
\\
\Ecal^D (s, \widehat{Z}^I \psi)^{1/2}
+
\Ecal_1 (s, Z^I v)^{1/2}
&\leq C_1 \eps s^\delta,
\quad
&|I|  &\leq N.
\endaligned
\ee
In the above, the constant $C_1 \gg 1$ is to be determined, $\eps \ll 1$ measures the size the initial data, and we let $C_1 \eps \ll 1$, and $0<\delta \leq \tfrac{1}{10}$. \underline{For the rest of section \ref{sec:proof2}} we assume, without restating the fact, that \eqref{eq:BA-Dirac2} hold on a hyperboloidal time interval $[s_0, s_1)$ where $s_1$ is defined as
$$
s_1 \define \sup \{ s: s>s_0,\, \eqref{eq:BA-Dirac}\,\, \text{holds} \}.
$$

Similar to Propositions \ref{prop:L2}, \ref{prop:Linfty}, and \ref{prop:psi-t}, we have the following preliminary $L^2$ and $L^\infty$ estimates.

\begin{proposition}\label{prop:L2-2}
For $s \in [s_0, s_1)$ we have
$$
\aligned
\big\| (s/t) \widehat{Z}^I \psi \big\|_{L^2_f(\Hcal_s)}
+
\big\| (s/t) Z^I \psi \big\|_{L^2_f(\Hcal_s)}
+
\big\| (\widehat{Z}^I \psi)_- \big\|_{L^2_f(\Hcal_s)}
&\lesssim
\left\lbrace
\begin{array}{ll}
C_1 \eps,& |I|\leq N-1,
\vspace{0.15cm}
\\
C_1 \eps s^\delta, & |I|\leq N,
\end{array}
\right.
\\
\big\| (s/t) \del Z^I v \big\|_{L^2_f(\Hcal_s)}
+
\big\| (s/t)  Z^I  \del v \big\|_{L^2_f(\Hcal_s)}
+
\big\| Z^I v \big\|_{L^2_f(\Hcal_s)}
&\lesssim
\left\lbrace
\begin{array}{ll}
C_1 \eps, & |I|\leq N-1,
\vspace{0.15cm}
\\
C_1 \eps s^\delta,&|I|\leq N.
\end{array}
\right.
\endaligned
$$
%
%$$
%\aligned
%\big\| (s/t) \widehat{Z}^I \psi \big\|_{L^2_f(\Hcal_s)}
%+
%\big\| (s/t) Z^I \psi \big\|_{L^2_f(\Hcal_s)}
%+
%\big\| (\widehat{Z}^I \psi)_- \big\|_{L^2_f(\Hcal_s)}
%&\lesssim
%C_1 \eps,
%&|I|  &\leq N-1,
%\\
%\big\| (s/t) \widehat{Z}^I \psi \big\|_{L^2_f(\Hcal_s)}
%+
%\big\| (s/t) Z^I \psi \big\|_{L^2_f(\Hcal_s)}
%+
%\big\| (\widehat{Z}^I \psi)_- \big\|_{L^2_f(\Hcal_s)}
%&\lesssim
%C_1 \eps s^\delta,
%&|I|  &\leq N,
%\\
%\big\| (s/t) \del Z^I v \big\|_{L^2_f(\Hcal_s)}
%+
%\big\| (s/t)  Z^I  \del v \big\|_{L^2_f(\Hcal_s)}
%+
%\big\| Z^I v \big\|_{L^2_f(\Hcal_s)}
%&\lesssim
%C_1 \eps,
%&|I|  &\leq N-1,
%\\
%\big\| (s/t) \del Z^I v \big\|_{L^2_f(\Hcal_s)}
%+
%\big\| (s/t)  Z^I \del v \big\|_{L^2_f(\Hcal_s)}
%+
%\big\| Z^I v \big\|_{L^2_f(\Hcal_s)}
%&\lesssim
%C_1 \eps s^\delta,
%&|I|  &\leq N.
%\endaligned
%$$
\end{proposition}

\begin{proposition}\label{prop:Linfty2}
For $s \in [s_0, s_1)$ we have
$$
\aligned
\big| \widehat{Z}^I \psi \big| 
+
\big| Z^I \psi \big| 
+
(t/s)\big| (\widehat{Z} \psi)_- \big| 
&\lesssim
\left\lbrace
\begin{array}{ll}
C_1 \eps s^{-1}, & |I|  \leq N-3,
\vspace{0.15cm}
\\
C_1 \eps s^{-1+\delta}, & |I|  \leq N-2,
\end{array}
\right.
\\
\big| \del Z^I v \big|
+
\big| Z^I \del v \big|
+
(t/s)\big| Z^I v \big|
&\lesssim
\left\lbrace
\begin{array}{ll}
C_1 \eps s^{-1}, & |I|  \leq N-3,
\vspace{0.15cm}
\\
C_1 \eps s^{-1+\delta}, & |I|  \leq N-2.
\end{array}
\right.
\endaligned
$$

%$$
%\aligned
%\big| \widehat{Z}^I \psi \big| 
%+
%\big| Z^I \psi \big| 
%+
%(t/s)\big| (\widehat{Z} \psi)_- \big| 
%&\lesssim
%C_1 \eps s^{-1},
%\qquad
%&|I|  \leq N-3,
%\\
%\big| \widehat{Z}^I \psi \big| 
%+
%\big| Z^I \psi \big| 
%+
%(t/s) \big| (\widehat{Z}^I \psi)_- \big| 
%&\lesssim
%C_1 \eps s^{-1+\delta},
%\qquad
%&|I|  \leq N-2,
%\\
%\big| \del Z^I v \big|
%+
%\big| Z^I \del v \big|
%+
%(t/s)\big| Z^I v \big|
%&\lesssim
%C_1 \eps s^{-1},
%\qquad
%&|I|  \leq N-3,
%\\
%\big| \del Z^I v \big|
%+
%\big| Z^I \del v \big|
%+
%(t/s) \big| Z^I v \big|
%&\lesssim
%C_1 \eps s^{-1+\delta},
%\qquad
%&|I|  \leq N-2.
%\endaligned
%$$
\end{proposition}

\begin{proposition}\label{prop:psi-t2}
The following weighted $L^2$-estimates are valid for $s \in [s_0, s_1)$
$$
\aligned
\big\| (t-r) (s/t) \del Z^I \psi \big\|_{L^2_f(\Hcal_s)}
+
\big\| (t-r) (s/t) \del \widehat{Z}^I \psi \big\|_{L^2_f(\Hcal_s)}
&\lesssim
C_1 \eps s^{\delta},
\quad |I|  \leq N-1,
\endaligned
$$
and the following pointwise estimates also hold for $s \in [s_0, s_1)$
$$
\aligned
\big| \del Z^I \psi \big|
+
\big| \del \widehat{Z}^I \psi \big|
&\lesssim
C_1 \eps (t-r)^{-1} s^{-1+\delta},
\quad |I| \leq N-3.
\endaligned
$$
\end{proposition}

\subsection{Improved estimates for the Klein-Gordon field}
In order to improve the energy bounds for the Klein-Gordon field, we apply two different arguments for the lower-order energy case and for the top-order energy case. For the lower-order case, we rely on a nonlinear transformation (of Type 1 in section \ref{subsec:hidden-null}) to remove the slowly-decaying term $\psi^* \gamma^0 \psi$. This is at the expense of introducing null and cubic terms yet nevertheless allows us to obtain uniform energy bounds. 

On the other hand, when deriving the refined bound for the top-order Klein-Gordon energy the nonlinear transformation is invalid due to issues with regularity. Thus in this case we need to utilise the hidden Klein-Gordon structure of the nonlinearities as shown in Lemma \ref{lem:hidden-KG} and Lemma \ref{lem:L-Dirac}. Using this we can improve the energy bounds with the aid of the linear behavior of $\psi$ in the lower-order case.

\begin{lemma}\label{lem:transf-v2}
Let $\widetilde{v} \define v - \psi^* \gamma^0 \psi$. Then $\widetilde{v}$ solves the following Klein-Gordon equation
\bel{eq:KG-new2}
-\Box \widetilde{v} + \widetilde{v}
= -i \del_\nu(v\psi^*) \gamma^\nu \gamma^0 \psi +  i \psi^* \gamma^0 \gamma^\nu \del_\nu (v\psi) + 2 \eta^{\alpha\beta}\del_\alpha \psi^* \gamma^0 \del_\beta \psi.
\ee
\end{lemma}
\begin{proof}
The proof is straightforward. 
%We act the Klein-Gordon operator to $\widetilde{v}$ to obtain
%$$
%\aligned
%-\Box \widetilde{v} + \widetilde{v}
%=
%-\Box \big(v - \psi^* \gamma^0 \psi\big) + \big(v - \psi^* \gamma^0 \psi\big)
%=
%-\Box v + v - \psi^* \gamma^0 \psi + \Box \big( \psi^* \gamma^0 \psi \big).
%\endaligned
%$$
%We also have
%$$
%\aligned
%-\Box \psi
%=
%-i\gamma^\nu \del_\nu \big( i\gamma^mu \del_\mu \psi \big)
%=
%i\gamma^\nu \del_\nu (v \psi).
%\endaligned
%$$
%Finally recalling the original equations in \eqref{eq:D-KG} leads us to \eqref{eq:KG-new2}.

\end{proof}

\begin{lemma}\label{lem:unif-est-tildev2}
%Under the bootstrap assumptions in \eqref{eq:BA-Dirac2}, the following estimates are valid for $s \in [s_0, s_1)$
We have
$$
\aligned
\Ecal_1 (s, Z^I \widetilde{v})^{1/2}
&\lesssim \eps + (C_1 \eps)^2,
\qquad
&|I|  \leq N-1.
\endaligned
$$
\end{lemma}
\begin{proof}
Acting $Z^I$ with $|I|  \leq N-1$ on equation \eqref{eq:KG-new2} produces
$$
-\Box Z^I \widetilde{v} + Z^I \widetilde{v}
=Z^I \big(- i \del_\nu(v\psi^*) \gamma^\nu \gamma^0 \psi +  i \psi^* \gamma^0 \gamma^\nu \del_\nu (v\psi) + 2 \del_\alpha \psi^* \gamma^0 \del^\alpha \psi\big).
$$
The energy estimates of Proposition \ref{prop:energy-ineq-KG} then imply
$$
\aligned
\Ecal_1 (s, Z^I \widetilde{v})^{1/2}
&\lesssim
\Ecal_1 (s_0, Z^I \widetilde{v})^{1/2}
\\
&+
\int_{s_0}^s \Big\|  Z^I \big( -i \del_\nu(v\psi^*) \gamma^\nu \gamma^0 \psi +  i \psi^* \gamma^0 \gamma^\nu \del_\nu (v\psi) + 2 \del_\alpha \psi^* \gamma^0 \del^\alpha \psi\big)  \Big\|_{L^2_f(\Hcal_\tau)} \, \di\tau.
\endaligned
$$

The proof follows similar to  Lemma \ref{lem:NL-KGt}, where we bound each of the terms to get the desired estimates.
\end{proof}

The following lemma is the key to closing the top-order bootstraps for the Klein-Gordon field.
\begin{lemma}\label{lem:est-Fv2}
%Let the estimates in \eqref{eq:BA-Dirac2} hold, then for $s \in [s_0, s_1)$ we have
We have
$$
\aligned
\| Z^I (\psi^* \gamma^0 \psi) \|_{L^2_f(\Hcal_s)} 
&\lesssim 
(C_1 \eps)^2 s^{-1+\delta},\qquad |I|\leq N.
\endaligned
$$
\end{lemma}

\begin{proof}

By Lemma \ref{lem:L-Dirac} we find
$$
\big| Z^I \big(\psi^* \gamma^0 \psi \big) \big|
\leq
\sum_{\substack{|I_1|+|I_2|=N}}\big|\widehat{Z}^{I_1} \psi \big)^* \gamma^0  \big( \widehat{Z}^{I_2} \psi \big|.
$$
Next we apply Lemma \ref{lem:hidden-KG} to reveal the hidden Klein-Gordon structure of the nonlinearity:
$$
\aligned
Z^I \big(\psi^* \gamma^0 \psi \big)
&=
{1\over 4}\sum_{\substack{|I_1|+|I_2|=N}} \Big(  \big(\widehat{Z}^{I_1}  \psi \big)_-{}^* \gamma^0  \big( \widehat{Z}^{I_2} \psi \big)_-
+
\big(\widehat{Z}^{I_1}  \psi \big)_-{}^* \gamma^0  \big( \widehat{Z}^{I_2}  \psi \big)_+
\\
&\hskip2cm+
\big(\widehat{Z}^{I_1} \psi \big)_+{}^* \gamma^0  \big( \widehat{Z}^{I_2} \psi \big)_-
+
\Big( \frac{\tau}{t}\Big)^2\big(\widehat{Z}^{I_1}  \psi \big)^* \gamma^0  \big( \widehat{Z}^{I_2}  \psi \big)
  \Big).
\endaligned
$$
We recall that $\big(\widehat{Z}^{I_1}  \psi \big)_-$ can be regarded as a Klein-Gordon component in the sense that it enjoys the same $L^2$--type and $L^\infty$ estimates as Klein-Gordon components, while $\big(\widehat{Z}^{I_1}  \psi \big)_+$ enjoys the same good bounds as $\widehat{Z}^{I_1}  \psi$. 
We proceed to bound
$$
\aligned
\big\|  &Z^I \big(\psi^* \gamma^0 \psi \big)  \big\|_{L^2_f(\Hcal_s)}
\\
\lesssim
&\sum_{\substack{|I_1|+|I_2|=N}}\Big( \big\| \big(\widehat{Z}^{I_1}  \psi \big)_-{}^* \gamma^0  \big( \widehat{Z}^{I_2}  \psi \big)_- \big\|_{L^2_f(\Hcal_s)}
+
\big\| \big(\widehat{Z}^{I_1}  \psi \big)_-{}^* \gamma^0  \big( \widehat{Z}^{I_2}  \psi \big)_+ \big\|_{L^2_f(\Hcal_s)}
\\
&\hskip2.5cm+
%\big\| \big(\widehat{Z}^{I_1}  \psi \big)_+{}^* \gamma^0  \big( \widehat{Z}^{I_2}  \psi \big)_- \big\|_{L^2_f(\Hcal_s)}
%+
\big\| (s/t)^2 \big(\widehat{Z}^{I_1}  \psi \big)^* \gamma^0  \big( \widehat{Z}^{I_2}  \psi \big) \big\|_{L^2_f(\Hcal_s)}
  \Big).
\endaligned
$$
We first show
$$
\aligned
\sum_{\substack{|I_1|+|I_2|=N}}
&\big\| \big(\widehat{Z}^{I_1} \psi \big)_-{}^* \gamma^0  \big( \widehat{Z}^{I_2}  \psi \big)_- \big\|_{L^2_f(\Hcal_s)}
\\
&\lesssim
\sum_{\substack{|I_1|\leq N-3\\ |I_2|\leq N}}  \big\| \big(\widehat{Z}^{I_1}  \psi \big)_- \big\|_{L^\infty(\Hcal_s)} \big\| \big( \widehat{Z}^{I_2}  \psi \big)_- \big\|_{L^2_f(\Hcal_s)}
+
\sum_{\substack{|I_1|\leq N-2\\ |I_2|\leq N-1}}  \big\| \big(\widehat{Z}^{I_1}  \psi \big)_- \big\|_{L^\infty(\Hcal_s)} \big\| \big( \widehat{Z}^{I_2}  \psi \big)_- \big\|_{L^2_f(\Hcal_s)}
\\&
\lesssim (C_1 \eps)^2 s^{-1+\delta},
\endaligned
$$
in which the assumption $N\geq 4$ was used in the first inequality.
Similarly, we also have
\begin{align*}
\sum_{\substack{|I_1|+|I_2|=N}}
&\big\| \big(\widehat{Z}^{I_1}  \psi \big)_-{}^* \gamma^0  \big( \widehat{Z}^{I_2} \psi \big)_+ \big\|_{L^2_f(\Hcal_s)}
\\
\lesssim
&\sum_{\substack{|I_1|\leq N-3\\ |I_2|\leq N}} \big\| (t/s) \big(\widehat{Z}^{I_1} \psi \big)_-\big\|_{L^\infty(\Hcal_s)} \big\| (s/t) \big( \widehat{Z}^{I_2}  \psi \big)_+ \big\|_{L^2_f(\Hcal_s)}
\\
&\quad +\sum_{\substack{|I_1|\leq N-2\\ |I_2|\leq N-1}} \big\| (t/s) \big(\widehat{Z}^{I_1} \psi \big)_-\big\|_{L^\infty(\Hcal_s)} \big\| (s/t) \big( \widehat{Z}^{I_2} \psi \big)_+ \big\|_{L^2_f(\Hcal_s)}
\\
&\quad +
\sum_{\substack{|I_1|\leq N\\ |I_2|\leq N-3}} \big\| \big(\widehat{Z}^{I_1} \psi \big)_-\big\|_{L^2_f(\Hcal_s)} \big\| (s/t) \big( \widehat{Z}^{I_2} \psi \big)_+ \big\|_{L^\infty(\Hcal_s)}
\\
&\quad +
\sum_{\substack{|I_1|\leq N-1\\ |I_2|\leq N-2}} \big\| \big(\widehat{Z}^{I_1}  \psi \big)_-\big\|_{L^2_f(\Hcal_s)} \big\| (s/t) \big( \widehat{Z}^{I_2}  \psi \big)_+ \big\|_{L^\infty(\Hcal_s)}
%\\
%\lesssim
%&\sum_{\substack{|I_1|+|J_1|\leq N-3\\ |I_2|+|J_2|\leq N}} \big\| (t/s) \big(\del^{I_1} \widehat{L}^{J_1} \psi \big)_-\big\|_{L^\infty(\Hcal_s)} \big\| (s/t) \del^{I_2} \widehat{L}^{J_2} \psi  \big\|_{L^2_f(\Hcal_s)}
%\\
%&\quad +
%\sum_{\substack{|I_1|+|J_1|\leq N\\ |I_2|+|J_2|\leq N-3}} \big\| \big(\del^{I_1} \widehat{L}^{J_1} \psi \big)_-\big\|_{L^2_f(\Hcal_s)} \big\| (s/t)  \del^{I_2} \widehat{L}^{J_2} \psi \big\|_{L^\infty(\Hcal_s)}
\\
\lesssim
& (C_1 \eps)^2 s^{-1+\delta}.
\end{align*}
We then estimate
$$
\aligned
& \sum_{\substack{|I_1|+|I_2|=N}}
\big\| (s/t)^2 \big(\widehat{Z}^{I_1}  \psi \big)^* \gamma^0  \big( \widehat{Z}^{I_2}  \psi \big) \big\|_{L^2_f(\Hcal_s)}
\\
\lesssim
&\sum_{\substack{|I_1|\leq N-3\\ |I_2|\leq N}}
\big\| (s/t) \widehat{Z}^{I_1}  \psi  \big\|_{L^\infty(\Hcal_s)}   \big\| (s/t) \big( \widehat{Z}^{I_2}  \psi \big) \big\|_{L^2_f(\Hcal_s)}
+ \sum_{\substack{|I_1|\leq N-2\\ |I_2|\leq N-1}}
\big\| (s/t) \widehat{Z}^{I_1}  \psi  \big\|_{L^\infty(\Hcal_s)}   \big\| (s/t) \big( \widehat{Z}^{I_2} \psi \big) \big\|_{L^2_f(\Hcal_s)}
\\
\lesssim
& (C_1 \eps)^2 s^{-1+\delta}.
\endaligned
$$

Gathering the above estimates, we obtain
$$
\big\| Z^I \big(\psi^* \gamma^0 \psi \big)  \big\|_{L^2_f(\Hcal_s)}
\lesssim (C_1 \eps)^2 s^{-1+\delta},
\qquad
|I|+|J| \leq N.
$$
\end{proof}

\begin{proposition}\label{prop:KG-improved2}
%Assuming the estimates in \eqref{eq:BA-Dirac2} hold, for $s \in [s_0, s_1)$ 
We have
$$
\aligned
\Ecal_1 (s, Z^I v)^{1/2}
&\lesssim 
\left\lbrace
\begin{array}{ll}
\eps + (C_1 \eps)^2, & |I|  \leq N-1,
\vspace{0.15cm}
\\
\eps + (C_1 \eps)^2 s^\delta, & |I|  \leq N.
\end{array}
\right.
\endaligned
$$

%$$
%\aligned
%\Ecal_1 (s, Z^I v)^{1/2}
%&\lesssim \eps + (C_1 \eps)^2,
%\qquad
%&|I|  \leq N-1,
%\\
%\Ecal_1 (s, Z^I v)^{1/2}
%&\lesssim  \eps + (C_1 \eps)^2 s^\delta,
%\qquad
%&|I|  \leq N.
%\endaligned
%$$
\end{proposition}
\begin{proof}

We first show the improved energy estimates in the case of $|I|  \leq N$.
We act the Klein-Gordon equation in \eqref{eq:D-KG} with $Z^I$ to get
$$
-\Box Z^I v + Z^I v 
= Z^I \big(\psi^* \gamma^0 \psi \big),
$$
The energy estimates of Proposition \ref{prop:energy-ineq-KG} and the key result of Lemma \ref{lem:est-Fv2} imply
$$
\aligned
\Ecal_1 (s, Z^I \widetilde{v})^{1/2}
&\lesssim
\Ecal_1 (s_0, Z^I \widetilde{v})^{1/2}
+ \int_{s_0}^s \big\| Z^I \big(\psi^* \gamma^0 \psi \big)  \big\|_{L^2_f(\Hcal_\tau)} \, \di \tau
\lesssim
\eps + (C_1 \eps)^2  \int_{s_0}^s  \tau^{-1+\delta} \, \di \tau
\\
&\lesssim
\eps + (C_1 \eps)^2 s^\delta.
\endaligned
$$

We next turn to the uniform energy bounds for $|I| \leq N-1$. Due to the uniform estimates of Lemma \ref{lem:unif-est-tildev2}, we just need to study the difference between $v$ and $\widetilde{v}$. This is a quadratic term $\psi^* \gamma^0 \psi$ which, for $|I| \leq N-1$, is controlled using Lemma \ref{lem:est-Fv2} as
$$
\aligned
\Ecal_1 \big(s, &\, Z^I \big(\psi^* \gamma^0 \psi \big)\big)^{1/2}
\\
&\lesssim
\big\| (s/t) \del_t Z^I \big(\psi^* \gamma^0 \psi \big)  \big\|_{L^2_f(\Hcal_s)}
+
\sum_a \big\| \underline{\del}_a Z^I \big(\psi^* \gamma^0 \psi \big)  \big\|_{L^2_f(\Hcal_s)}
+
\big\| Z^I \big(\psi^* \gamma^0 \psi \big)  \big\|_{L^2_f(\Hcal_s)}
\\
&\lesssim  (C_1 \eps)^2 s^{-1+\delta}.
\endaligned
$$
In conclusion we find, for $|I|\leq N-1$,
$$
\Ecal_1(s, Z^I v)^{1/2}
\lesssim
\Ecal_1(s, Z^I \widetilde{v})^{1/2}
+
\Ecal_1\big(s, Z^I \big(\psi^* \gamma^0 \psi \big)\big)^{1/2}
\lesssim
\eps+ (C_1 \eps)^2.
$$

\end{proof}

\subsection{Improved estimates for the Dirac field}
In order to improve the energy bounds for the Dirac field, we also use two different arguments for the lower-order energy case and for the top-order energy case. For the lower-order case, our strategy is to introduce the new variable
$$
\widetilde{\psi} = \psi + i\gamma^\nu \del_\nu (v F \psi),
$$
and derive a uniform energy bound for its lower-order energy. This is a  nonlinear transformation of Type 3 in Section \ref{subsec:hidden-null} and it allows us to remove the slowly-decaying nonlinearity $v \psi$ at the expense of introducing null and cubic terms.  After obtaining lower-order uniform energy bounds for $\widetilde{\psi}$ we can then easily get improved bounds for the lower-order energy of $\psi$ since the difference between $\psi, \widetilde{\psi}$ is a quadratic term.

Similar to the strategy for the Klein-Gordon field, this transformation to $\widetilde{\psi}$ is not valid at top-order. Nevertheless with the linear behavior of the fields $\psi, v$ in the bootstrap assumptions \eqref{eq:BA-Dirac2}, we can also close the bootstrap for the top-order energy estimates.

\begin{lemma}\label{lemma:EoM-tilde-psi2}
Let $\widetilde{\psi} \define \psi + i\gamma^\nu \del_\nu (v F \psi)$. Then $\widetilde{\psi}$ solves the following Dirac equation
\bel{eq:Dirac-new2}
-i\gamma^\mu \del_\mu \widetilde{\psi}
=
\big(\psi^* \gamma^0 \psi\big) F\psi + i \gamma^\nu v \del_\nu (v F\psi) - 2\del_\alpha v F \del^\alpha \psi.
\ee
\end{lemma}
\begin{proof}
The proof is straightforward. 
%We act the Dirac operator on $\widetilde{\psi}$ to get
%$$
%\aligned
%-i\gamma^\mu \del_\mu \widetilde{\psi}
%&=
%-i\gamma^\mu \del_\mu \big(  \psi + i\gamma^\nu \del_\nu (v A\psi) \big)
%=
%-i\gamma^\mu \del_\mu \psi - \Box (v A\psi)
%\\
%&=
%-i\gamma^\mu \del_\mu \psi + (-\Box v) A\psi + vA(-\Box \psi) - 2 \eta^{\alpha\beta}\del_\alpha v A\del_\beta \psi.
%\endaligned
%$$
%We recall
%$$
%\aligned
%-\Box \psi
%=
%-i\gamma^\nu \del_\nu \big( i\gamma^\mu \del_\mu \psi \big)
%=
%i\gamma^\nu \del_\nu (v A\psi).
%\endaligned
%$$
%Thus we arrive at the desired result by using the original equations in \eqref{eq:D-KG}.
\end{proof}

\begin{lemma}\label{lem:unif-est-tildepsi2}
Let the estimates in \eqref{eq:BA-Dirac2} hold, then for $s \in [s_0, s_1)$ we have
$$ 
\Ecal^D (s, \widehat{Z}^I  \widetilde{\psi})^{1/2}
\lesssim \eps + (C_1 \eps)^2,
\qquad
|I|  \leq N-1.
$$
\end{lemma}

\begin{proof}
From Proposition \ref{prop:energy-ineq-Dirac} we see that we need to control
$$
\sum_{|I|\leq N-1} \int_{s_0}^s \| \widehat{Z}^I  (i\gamma^\mu \del_\mu \widetilde{\psi}) \|_{L^2_f(\Hcal_\tau)} \di \tau.
$$
Mimicking the analysis in Lemma \ref{lem:NL-Dt}, we can show (recall Lemma \ref{lemma:EoM-tilde-psi2}) that
\be \notag
\sum_{|I|\leq N-1} \| \widehat{Z}^I  (i\gamma^\mu \del_\mu \widetilde{\psi}) \|_{L^2_f(\Hcal_\tau)} 
\lesssim
(C_1 \eps)^2 \tau^{-2+2\delta},
\ee
which leads us to
$$
\Ecal^D (s, \widehat{Z}^I \widetilde{\psi})^{1/2}
\lesssim 
\Ecal^D (s_0, \widehat{Z}^J \widetilde{\psi})^{1/2}
+
(C_1\eps)^2 \int_{s_0}^s \tau ^{-2+2\delta}\di \tau 
\lesssim
\eps + (C_1\eps)^2
\qquad
|I|\leq N-1.
$$
\end{proof}

\begin{proposition}\label{prop:Dirac-improved2}
Let the estimates in \eqref{eq:BA-Dirac2} hold, then for $s \in [s_0, s_1)$ we have
$$
\aligned
\Ecal^D (s, \widehat{Z}^I \psi)^{1/2}
&\lesssim
\left\lbrace
\begin{array}{ll}
\eps + (C_1 \eps)^2, & |I| \leq N-1,  
\vspace{0.15cm}
\\
\eps + (C_1 \eps)^2 s^\delta, & |I| \leq N.
\end{array} \right.
\endaligned
$$

%$$
%\aligned
%\Ecal^D (s, \widehat{Z}^I \psi)^{1/2}
%&\lesssim \eps + (C_1 \eps)^2,
%\qquad
%&|I| \leq N-1,  
%\\
%\Ecal^D (s, \widehat{Z}^I \psi)^{1/2}
%&\lesssim \eps + (C_1 \eps)^2 s^\delta,
%\qquad
%&|I| \leq N.
%\endaligned
%$$
\end{proposition}

\begin{proof}
We begin with the estimate at top-order. For $|I|  \leq N$, and given $N \geq 4$, we have
$$
\aligned
\big\|& \widehat{Z}^I \big(v F\psi\big)  \big\|_{L^2_f(\Hcal_s)}
\\ & \lesssim
\sum_{\substack{|I_1|\leq N \\ |I_2|\leq N-3}} \| Z^{I_1}  v \|_{L^2_f(\Hcal_s)} \|\widehat{Z}^{I_2} \psi\|_{L^\infty(\Hcal_s)} 
+ \sum_{\substack{|I_1|\leq N-3 \\ |I_2|\leq N}} \| (t/s) Z^{I_1} v \|_{L^\infty(\Hcal_s)} \|(s/t) \widehat{Z}^{I_2} \psi\|_{L^2_f(\Hcal_s)} 
\\&\quad
+ \sum_{\substack{|I_1|\leq N-1 \\ |I_2|\leq N-2}} \| Z^{I_1}  v \|_{L^2_f(\Hcal_s)} \|\widehat{Z}^{I_2} \psi\|_{L^\infty(\Hcal_s)} 
+ \sum_{\substack{|I_1|\leq N-2 \\ |I_2|\leq N-1}} \| (t/s) Z^{I_1}  v \|_{L^\infty(\Hcal_s)} \|(s/t)\widehat{Z}^{I_2} \psi\|_{L^2_f(\Hcal_s)} 
\\&\lesssim
(C_1\eps)^2 s^{-1+\delta}.
\endaligned
$$
Note in the final step we carefully used the uniform energy bounds from Proposition \ref{prop:L2-2} and the sharp pointwise estimates from Proposition \ref{prop:Linfty2} so as not to pick up an $s^{2\delta}$ growth. 

Thus the energy inequality of Proposition \ref{prop:energy-ineq-Dirac} implies
$$
\aligned
\Ecal^D (s, \widehat{Z}^I \psi)^{1/2}
&\lesssim 
\Ecal^D (s_0, \widehat{Z}^I \psi)^{1/2}
+ \int_{s_0}^s \big\| \widehat{Z}^I \big(v F\psi\big)  \big\|_{L^2_f(\Hcal_\tau)} \, \di \tau
\\
&\lesssim
\eps + (C_1 \eps)^2\int_{s_0}^s \tau^{-1+\delta} \, d\tau
\lesssim
\eps + (C_1 \eps)^2 s^\delta.
\endaligned
$$

As for the case of $|I|\leq N-1$, we can show (similar to the proof of Proposition \ref{prop:low-D}), that
$$
\Ecal^D(s, \widehat{Z}^I \psi)^{1/2}
\lesssim
\Ecal^D(s, \widehat{Z}^I \widetilde{\psi})^{1/2}
+
\Ecal_1\big(s, \widehat{Z}^I \big(i\gamma^v\del_\nu(v F\psi)\big)\big)^{1/2}
\lesssim
\eps+ (C_1 \eps)^2.
$$

\end{proof}

\subsection{Proof of Theorem \ref{thm:main2}}

\begin{proof}[Proof of Theorem \ref{thm:main2}.]
The results of Propositions \ref{prop:KG-improved2} and \ref{prop:Dirac-improved2} imply that for a fixed $0<\delta\ll1$ and $\mathbb{N}\ni N \geq 4$ there exists an $\eps_0>0$ sufficiently small, such that for all $0<\eps\leq \eps_0$ we have for all $s\in [s_0, s_1)$ 
\be\notag
\aligned
\Ecal^D (s, \widehat{Z}^I  \psi)^{1/2}
+
\Ecal_1 (s, Z^I v)^{1/2}
&\leq \tfrac12 C_1 \eps_0,
\quad
&|I| &\leq N-1,  
\\
\Ecal^D (s, \widehat{Z}^I \psi)^{1/2}
+
\Ecal_1 (s, Z^I v)^{1/2}
&\leq \tfrac12 C_1 \eps_0 s^\delta,
\quad
&|I|  &\leq N.
\endaligned
\ee
Similar to the argument in the proof of Theorem \ref{thm:main1}, we can deduce from the above that $s_1=+\infty$.  As for the time decay and scattering, the proof is very similar to that of Theorem \ref{thm:main1}, and so we omit the details.
\end{proof}

%==============================================================================================
\subsection*{Acknowledgements}
SD and ZW are grateful to Philippe LeFloch (Sorbonne) for introducing them to the hyperboloidal foliation method. The authors also thank Pieter Blue (Edinburgh), and SD thanks additionally Zhen Lei (Fudan), for their constant encouragement.

\end{document}